\documentclass[11pt,letterpaper,reqno]{amsart}
\usepackage{tikz}
\usepackage{hyperref}
\usepackage{amsmath}
\usepackage{amsthm}
\usepackage{amssymb}
\usepackage{subfig}

\newtheorem{theorem}{Theorem}
\newtheorem{example}{Example}
\newtheorem{problem}{Problem}
\newtheorem{question}{Question}
\newtheorem{lemma}{Lemma}

\newtheorem{corollary}{Corollary}
\newtheorem{proposition}{Proposition}
\newtheorem{observation}{Observation}

\begin{document}

\title[Jet schemes, Newton polygons and continued fractions]{Jet schemes, Newton polygons and continued fractions}

\author[{G. Abdallah}]{Ghadi Abdallah}
\address{Lebanese University, Kalma, Hadath Campus, Lebanon}  
\email{ghadi.abdallah.1@st.ul.edu.lb}

\author[M. Leyton-Alvarez]{Maximiliano Leyton-\'Alvarez} %
\address{Instituto de Matem\'aticas, Universidad de Talca, Talca, Chile} %
\email{mleyton@utalca.cl}

 \author[B. Mourad]{Bassam Mourad}
\address{Lebanese University, Faculty of Science I, Hadath Campus, Beirut, Lebanon}  \email{bmourad@ul.edu.lb}

\author[H. Mourtada]{Hussein Mourtada}
\address{Universit\'e Paris Cit\'e, Sorbonne Universit\'e, CNRS, IMJ-PRG, F-75013 Paris, France }
\email{hussein.mourtada@imj-prg.fr}

\date{\today}

\thanks{The second author author was partially supported by Fondecyt project 1221535 and  Fondecyt Exploraci\'on project 13250049. The fourth author was partially supported by
the ANR-SINTROP}

\begin{abstract}
We study jet schemes of Newton non-degenerate plane curve singularities. We identify a subgraph of the graph of jet components and show that it can be constructed from walks on the lattice points in the first quadrant of the Cartesian plane. In particular, we determine all the irreducible components of the jet schemes. Furthermore, we prove that this subgraph encodes the embedded topological type of the curve singularity in the plane. Finally, we introduce a generating series defined in terms of the irreducible components of the jet schemes and their (co-)dimensions, and we prove that this series is rational and explicitly determines its poles.
\end{abstract}
\maketitle
\tableofcontents

\section{Introduction} \label{intro} 

The main objects studied in this article are the jet schemes of plane singular curves. For every integer $m \in \mathbf{N},$ the $m$-th jet scheme of a singular variety $X$ (here defined over an algebraically closed field $k$), denoted by $X_m,$ provides a finite-dimensional approximation of the arc space of $X.$ The arc space, denoted by $X_\infty,$ is the moduli space parametrizing arcs traced on $X;$ in other words, each point of $X_\infty$ corresponds to an arc on $X.$ In a certain sense, the geometry of $X_m$ is equivalent (up to a trivial fibration) to the geometry of the family of arcs lying in a smooth ambient space containing $X,$ whose "order of contact" with $X$ is greater than or equal to $m.$\\

Jet schemes and arc spaces are known to encode a great deal of information about singularities. Through jet schemes, one can detect rational complete intersection singularities and compute the log canonical threshold of a pair \cite{M1,M2}. They also lie at the heart of the mechanism of geometric motivic integration \cite{DL_mot}, and they allow for the definition of the Igusa motivic zeta function (the motivic analogue of the Igusa zeta function in the $p$-adic setting) \cite{DL_Igusa}. Although their structure is generally rather intricate, the guiding philosophy for studying them can be summarized as follows \cite{M_hand}:\\

\begin{center}
    
"Arc spaces and jet schemes can transform a difficult problem concerning
a relatively simple object into a relatively simple problem concerning a difficult object."
\end{center}

Their structure has been understood for several classes of singularities, \emph{e.g.}, \cite{M_eq, M_rdp, DN, M_ts, CM,MVV}. A distinctive feature in the study of jet schemes is that, in order to extract information about singularities, one must consider them in a dynamical way: letting $m$ vary and carefully analyzing how the geometry of the jet schemes changes with $m.$  This was one of the motivation  behind introducing the concept known as  \textit{the graph of jet components} \cite{M_eq,M_ts}. This is a leveled graph which for a variety $X$ is defined as follows:
Its vertices at the level $m$ correspond to the irreducible components of $X_m$  and an edge is drawn from a vertex $v'$ of order $m+1$ to a vertex $v$ of order $m$
whenever the component corresponding to $v'$ maps into the one corresponding to $v$ under the natural truncation map $X_{m+1}\longrightarrow X_m.$ It has been shown, for certain classes of singularities, that the data of this graph (in fact, already of a suitable subgraph) can completely determine the embedded topological type of a singularity \cite{M_eq, M_ts}, and in some cases even its analytical type \cite{CM}. Moreover, it is closely related to the problem of resolution of singularities and to the embedded Nash problem \cite{dFSh1, dFSh2, M_rdp, MP, KMPT, M_g, B_al, K}.

The main purpose of this article is to determine the irreducible components of the jet schemes of a Newton non-degenerate plane curve singularity, and to identify a subgraph of the graph of jet components that encodes the embedded topological type of the singularity. Recall that Newton non-degenerate plane curve singularities are those that can be resolved by a toric morphism whose target is the plane. In the case of an irreducible singularity, we show that this problem is related to the following problem.
\begin{problem}    
 Let $p,q$ be two coprime positive integers with $0<p<q.$ Denote by $L_{(p,q)}$ the ray emanating from the origin of the plane in the direction of the point $(p,q),$ i.e., the half-line passing through $(p,q)$ with slope $q/p.$  Let $\alpha\in L_{(p,q)}\cap \mathbf{N}^2_0$. Determine the walk in the plane traced by the broken line (i.e. union of line segments) which starts at the point $\alpha+(1,1)$ and, at each step, it moves towards $L_{(p,q)}$ until reaching it. 
\end{problem}
See Figure \ref{walk(2,3)} for a graphical intuition of the problem. We show that this walk (uniquely determined by the problem) can be described by a special type of continued fraction defined via iterated Euclidean divisions. The subgraph of the graph of jet components that we seek is “almost” precisely given by the walks described above (as $\alpha$ varies), where $p$ and $q$ are determined by the Newton polygon of the (irreducible) curve; the integers $p$ and $q$ are closely related to the Puiseux pairs. In the reducible case, the walk is slightly more involved, but we obtain an analogous result. In particular, we determine the irreducible components of all the jet schemes. In addition, we prove that the data of the graph of jet components determines the embedded topological type of the curve.\\

In the last part, for a plane curve singularity $\mathcal{C}$ as above, we consider the following generating series:

$$G=\sum_{K,m} u^{\text{codim}(K;\mathbf{A}^2_\infty)}\cdot v^m\in \mathbf{Z}[[u,v]]$$
where $m\geq 1$ and $K$ is an irreducible component in $\text{Cont}^m(\mathcal{C})^0\neq \varnothing.$ Here, $\text{Cont}^m(\mathcal{C})^0$ is very much related to the $(m-1)-$th jet scheme, see section 2 for more details. This series was considered first in the PhD thesis of the last author. 
We show that $G=G(\mathcal{C})$ is rational in $u,v$ (this means that $G\in \mathbf{Z}(u,v)$) and determine its poles. This series is very much related to other zeta functions associated with singularities \cite{DL_Igusa,LL,V} and their poles have a lot in common but may be different.\\

The article is organized as follows:

\begin{itemize}
	\item In Section \ref{sec:jets}, we define the jet schemes and arc spaces of a variety over an algebraically closed field $k$. The arc spaces, will in fact, be used in computing the irreducible components of the jet schemes (over the origin) of a Newton non-degenerate plane curve singularity.
	\item In Section \ref{stair}, we will describe a staircase algorithm that will determine the irreducible components of these jet schemes.
	\item Sections \ref{sec : newton-br} and \ref{sec:newton} are dedicated to defining Newton non-degenerate branches and curves in the plane $\mathbf{C}^2$.
	\item In Section \ref{sec:crucial}, we prove some auxiliary results that will be used in showing the main theorems in Section \ref{sec:theorem}. 
	\item Finally, in Section \ref{sec:sum}, we compute the following invariant: the generating series associated to the jet schemes over the origin of a Newton non-degenerate plane curve singularity. Then, we show that it is rational and we determine its "poles" as well. 
\end{itemize}

\section{Jet Schemes and Arc Spaces} \label{sec:jets} 
In this section, we introduce the jet schemes and arc space of a variety. Denote by $k$ an algebraically closed field. Denote by $\mathbf{Var}_k$ the category of varieties over $k$ and $\textbf{Set}$ that of sets. Let $X\rightarrow Spec(k)$ be such a variety. For $m\in \mathbf{N}_0$, denote by $\mathcal{X}_m:\mathbf{Var}_k\rightarrow \mathbf{Set}$, the functor defined by
$$\mathcal{X}_m(Z)=\mathbf{Var}_k[Z\times_k Spec(k[t]/\langle t^{m+1}\rangle), X].$$
This functor is representable \cite{Voj} by a variety $X_m$ over $k$ called the \emph{jet scheme} at the level $m$ of $X$. For $m\geq n$, there is a  morphism $\pi_{m,n}:X_m\rightarrow X_n$ called the \emph{truncation morphism} that is induced by the natural map of rings
$$k[t]/\langle t^{n+1}\rangle \rightarrow k[t]/\langle t^{m+1}\rangle.$$ 
These morphisms verify: $\pi_{m,n}\circ \pi_{n,p} = \pi_{m,p}$ for $m\geq n\geq p$.
The \emph{arc space} $X_\infty$ of $X$ is the limit object in $\textbf{Var}_k$ of the jet schemes $X_m$ along the truncation morphisms $\pi_{m,n}:X_m\rightarrow X_n$ for $m\geq n$. To illustrate these notions, we take the following example.

\begin{example}
\normalfont{In order to find the jet schemes of an affine hypersurfaces in $k^n=\mathbf{A}^n$ (using the arc space $\mathbf{A}^n_\infty$), we set $\mathcal{O}_{k^n}=k[\mathbf{x}]$ : $\mathbf{x}=(x_1,..., x_n)$ the ring of the affine space $k^n$ and we let $f \in \mathcal{O}_{k^n}$ define a radical ideal. Define $X = \{ f = 0\}\subset k^n$ to be the reduced algebraic set whose ring of functions is 
$$\mathcal{O}_{X} = \mathcal{O}_{k^n}/{\langle f \rangle}.$$ 
The ring of the arc space of $k^n$ is $\mathcal{O}_{\mathbf{A}^n_\infty}=k[\mathbf{x}^{(0)},\mathbf{x}^{(1)}, \mathbf{x}^{(2)},...]$  with variables $\mathbf{x}^{(j)} = (x^{(j)}_1, ... , x^{(j)}_n)$. Next, we introduce the formal power series: 
$$x_i(t) = \sum_{j=0}^\infty x_i^{(j)} t^j \in  \mathcal{O}_{\mathbf{A}^n_\infty}[[t]].$$ 
Then, we expand $f[x_1(t),..., x_n(t)]=\sum_{m\geq 0} f^{(m)}t^m$.
Now we define $F_m$ to be the algebraic set in $\mathbf{A}^n_\infty$
$$F_m=\{f^{(j)}=0\}_{ 0\leq j\leq m} \subset \mathbf{A}^n_\infty.$$
Setting $\mathbf{A}^n_m=\{\mathbf{x}^{(j)}=0\}_{j>m} \subset \mathbf{A}^n_\infty$, then the \emph{jet scheme} of $X$ at the level $m$ is $X_m=\mathbf{A}^n_m\cap F_m\subset \mathbf{A}^n_m$.}
\end{example}

 In this paper, we take $n=2$, $x=x_1$, $y=x_2$ and every chosen $f\in \mathcal{O}_{k^2}$ is reduced (i.e. with no square factors). We are mainly interested in the jet schemes at each level $m$ of $\mathcal{C}=\{f=0\}$ over the origin of $k^2$ i.e. $$\mathcal{C}^0_m = \mathbf{A}^2_m\cap F_m \cap \{ x^{(0)} = y^{(0)} = 0\} \subset \mathbf{A}^2_m.$$ 
For this reason, we shall introduce $\mathcal{H}^\alpha_\infty$ for $\alpha\in \mathbf{N}^2$,  to be 
$$\mathcal{H}^\alpha_\infty = \{x^{(i)}=y^{(j)}=0 : 0\leq i<\alpha_1, 0\leq j<\alpha_2\}\subset \mathbf{A}^2_\infty.$$
 Also, we shall denote $F^\alpha_m=\{x^{(\alpha_1)}\neq 0\} \cap \{y^{(\alpha_2)}\neq 0\}\cap F_m\cap \mathcal{H}^\alpha_\infty \subset \mathbf{A}^2_\infty$ and 
 $$\mathcal{H}^\alpha_m=\mathcal{H}^\alpha_\infty \cap \mathbf{A}^2_m\subset \mathbf{A}^2_m.$$
 
\noindent Moreover, the contact locus 
$\text{Cont}^m(\mathcal{C})^0\subset \mathbf{A}^2_\infty$ is defined as
$$\text{Cont}^m(\mathcal{C})^0 = \{x^{(0)}=y^{(0)}=f^{(0)}=...=f^{(m-1)}=0\}\cap \{f^{(m)}\neq 0\}.$$ 
If we embed $\mathbf{A}^2_{m-1}\subset \mathbf{A}^2_{\infty}$ naturally for all $m$, then one can write
$$\text{Cont}^m(\mathcal{C})^0 = \mathcal{C}^0_{m-1}\cap \{f^{(m)}\neq 0\}.$$
Thus, studying the irreducible components of $\mathcal{C}^0_m$ for every $m$ is equivalent to studying those of $\text{Cont}^m(\mathcal{C})^0$ for all $m\in \mathbf{N}$.
\section{Special Staircase in the Plane} \label{stair} 

To determine the irreducible components of the jet schemes in the main theorems, a staircase algorithm will be described in this section.

Let $p,q$ be two coprime positive integers such that $0<p<q.$ Let $L_{(p,q)}$ be the half-line which starts at the origin of the plane and whose slope is $q/p;$ it passes through the point $(p,q)$.  

We call \emph{elementary steps} the unit vectors $(1,0)$ (horizontal) and $(0,1)$ (vertical). A \emph{broken line} is a finite sequence of such elementary steps. In this section, we shall answer the following question:

\begin{question}\label{Comb}
\textit{Let $\alpha\in L_{(p,q)}\cap \mathbf{N}^2_0$. What is the walk in the plane obtained by the broken line which starts at the point $\alpha+(1,1)$ and at each step it moves towards the half-line $L_{(p,q)}$ until reaching it?} 
\end{question}

It's worth noting that starting at the point $(1,1)$, the walker is always attracted by the half-line $L_{(p,q)}$. However, because of the unit of the steps he is obliged to jump over $L_{(p,q)}$ sometimes before reaching it. Note that the case $0<q < p$ is symmetric where we replace vertical steps by horizontal ones and vice versa. To clarify our procedures, we take the following examples.

\begin{example}
	Take $p=2$, $q=3$ and $\alpha=(0,0)$. Then, clearly the walker moves along the following points: $(1,1),(1,2), (2,2)$, $(2,3)$ in that order. This is seen in Figure \ref{walk(2,3)}. Clearly, we can encode this fact from the following formula:
$$(1,1)+\mathbf{(0,1)+(1,0)+(0,1)}=(2,3).$$
\end{example}

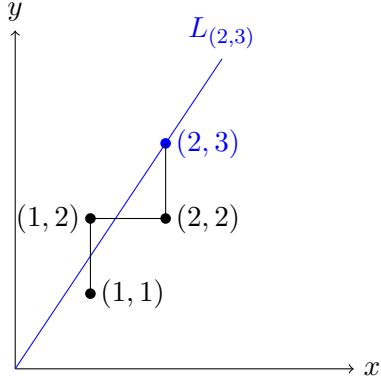
\begin{figure}[ht]
\centering
\begin{tikzpicture}
\draw[->] (0, 0) -- (4.5, 0) node[right] {$x$};
  \draw[->] (0, 0) -- (0, 4.5) node[above] {$y$};
  \draw[-] [scale=0.5, domain=0:5.5, smooth, variable=\x, blue] plot ({\x}, {1.5*\x}) node[above] {$L_{(2,3)}$};
  \draw[blue] (2,3) node[right] {$(2,3)$} ; 
  \fill[blue] (2,3) circle[radius=2pt];
  \draw (1,1) node[right] {$(1,1)$}; 
  \fill (1,1) circle[radius=2pt];
\draw (1,2) node[left] {$(1,2)$};
\fill (1,2) circle[radius=2pt]; 
\draw (2,2) node[right] {$(2,2)$}; 
\fill (2,2) circle[radius=2pt];
\draw[-] (1,1) -- (1,2);
\draw[-] (1,2) -- (2,2);
\draw[-] (2,2) -- (2,3);
  \end{tikzpicture}
  \caption{The walk mentioned in Question \ref{Comb} for $p=2$, $q=3$ and $\alpha=(0,0)$.} \label{walk(2,3)}
\end{figure}

\begin{example}
	Similarly to the previous example, if we take $p=2$, $q=3$ and $\alpha=(2,3)$, we find the same pattern for $(3,4)=(2,3)+(1,1)$:
$$(3,4)+\mathbf{(0,1)+(1,0)+(0,1)}=(4,6)$$
\end{example}

To answer Question \ref{Comb} for $0<p<q$ coprime, we start with $\alpha=(0,0)$ and introduce a
form of continued fractions which is given by the following iterated Euclidean divisions:

\begin{eqnarray}\label{Euclide}
    q & = & d_1p+r_1 \nonumber \\
    q+r_1 & = & d_2p+r_2  \nonumber\\
     &\vdots & \\
    q+r_{k-1} & = & d_kp+r_k  \nonumber\\
     &\vdots & \nonumber\\
     \nonumber
\end{eqnarray}

\noindent where in the first equality $d_1$ and $r_1$ are respectively the quotient and the remainder of the Euclidean division of $q$ by $p$. Recursively, for $k\geq 2,$ $d_k$ and $r_k$ are respectively the quotient and the remainder of the division of 
$q+r_{k-1}$ by $p$. In particular, we have $0\leq r_k<p.$ Once we obtain  $r_i=0$ for some index $i\geq 1,$ we stop
and we denote the continued fractions
$$\mathrm{SC}\left(\frac{q}{p}\right)=[[d_1,d_2,\ldots,d_i]].$$ 

Here $\mathrm{SC}$ stands for staircase; this name can be justified in Theorem \ref{walk}. Before proving the theorem, we need the following lemma.
\begin{lemma}\label{r_p}
With the same notation as above, we have the following.\\
  For $k=1,\ldots,p-1,$ $r_k>0$ and $r_p=0$ so that $$\mathrm{SC}\left(\frac{q}{p}\right)=[[d_1,\ldots,d_p]],$$
and  $q=d_1+\ldots+d_p.$
\end{lemma}

\begin{proof}
By summing up the first $k$ equations in (\ref{Euclide}), and then dividing by $p,$ we get
\begin{equation}\label{r_k}
\frac{kq}{p}=d_1+\cdots+d_k +\frac{r_k}{p}.
\end{equation}
Since $p$ and $q$ are coprime, then for $k<p,$ the left-hand side  of Equation (\ref{r_k}) is not an integer, and hence $r_k\not=0.$ However,
for $k=p,$ the left-hand side of Equation (\ref{r_k}) is an integer so that the right-hand side also is. This implies that $r_p=0$
$q=d_1+\ldots+d_p, $ and the proof is complete.

\end{proof}

As a result, we have the following theorem.

\begin{theorem}\label{walk}The walk which answers Question \ref{Comb} for $\alpha = (0,0)$ is determined by the following description:
\begin{itemize}
    \item The walker starts at $(1,1)$,
    \item For $1\leq i\leq p-1$, the walker moves vertically $d_i$  unit steps followed by one unit step horizontally,
    \item 
Finally,   the walker makes $d_p-1$  unit steps vertically.
\end{itemize}
\end{theorem}


\begin{proof}
Clearly, it is enough to prove that the walker moves along the following points:
\begin{center}
 $(1,1),(1,2),\ldots,(1,d_1+1),(2,d_1+1),$\\ 
  $(2,d_1+2),\ldots,(2,d_1+d_2+1),(3,d_1+d_2+1),$ \\  
  $\vdots$\\
   $(p-1,d_1+\cdots+d_{p-2}+1),\ldots,(p-1,d_1+\cdots+d_{p-1}+1),(p,d_1+\cdots+d_{p-1}+1)$\\
 $(p,d_1+\cdots+d_{p-1}+2),\ldots,(p,d_1+\cdots+d_{p-1}+d_p).$
\end{center}
As $p<q$, the line $L_{p,q}$ is above the point $(1,1),$ and hence the first step is vertical. To prove the same  for $k=1,\ldots,p-1,$ it is enough to show the following:

\begin{enumerate}
\item The point $(k,d_1+\cdots+d_k)$ is strictly below
$L_{p,q}.$
\item The point $(k,d_1+\cdots+d_k+1)$ is strictly above $L_{p,q}.$
\item The point $(k+1,d_1+\cdots+d_k+1)$ is strictly below $L_{p,q}.$
\item The point $(p,d_1+\cdots+d_{p-1}+d_p)$ belongs to $L_{p,q}$.
\end{enumerate}

Note that the last assertion follows from the equality $q=d_1+\ldots+d_p$ in  Lemma \ref{r_p}. Here the point being strictly below (respectively above) $L_{p,q}$ means that it belongs to the relative interior of the cone generated by $(1,0)^t$ and $(p,q)^t$ (respectively by $(p,q)^t$ and $(0,1)^t)$ where $t$ stands for transpose. The first assertion is equivalent to 
$\det\left[
\begin{array}{cc}
k & p \\
d_1+\cdots+d_k & q
\end{array}
\right] >0.$ Now by replacing $pd_i$ by $q+r_{i-1}-r_i,$ in (\ref{Euclide}), this latter determinant is equal to $$kq-pd_1-\cdots-pd_k=$$ $$kq-(q-r_1)-(q+r_1-r_2)-\cdots-(q+r_{k-1}-r_k)=r_k.$$ From Lemma $\ref{r_p}$, we know that $r_k>0,$ so that we obtain the first assertion. Similarly,
$$\det\left[
\begin{array}{cc}
p & k \\
q & d_1+\cdots+d_k+1
\end{array}
\right]=p-r_k>0$$ which gives the second assertion. Finally, the last assertion follows from the fact that $\det\left[
\begin{array}{cc}
k+1 & p \\
d_1+\cdots+d_k+1 & q
\end{array}
\right]=q-p+r_k>0$.
\end{proof}
Finally, the answer to Question \ref{Comb} for all $\alpha$ is given in the next corollary.
\begin{corollary} \label{th : equiv}
	Let $0<p<q$ be coprime integers and $\alpha\in L_{(p,q)}\cap \mathbf{N}^2_0$. Then, the answer to Question \ref{Comb} for both $(p,q,\alpha)$ and $(p,q,\alpha_0)$ with $\alpha_0=(0,0)$  is given by the same sequence of elementary steps.
\end{corollary}
\begin{proof}
Let $\alpha=a\cdot (p,q)$ with $a\geq 0$, and such that $\alpha+(1,1)=(ap+1,aq+1)$. Clearly, 
$$(p,q)=(1,1)+\sum_{i=1}^{p-1}[ d_i\cdot (0,1)+(1,0)]+(d_p-1)(0,1).$$
As $\alpha\in L_{(p,q)},$ then obviously $\alpha+(p,q)\in L_{(p,q)}.$  Thus,
$$\alpha+(1,1)+\sum_{i=1}^{p-1}[ d_i\cdot (0,1)+(1,0)]+(d_p-1)(0,1)=\alpha+(p,q).$$
This concludes the proof.
\end{proof}

\section{Jet Schemes of Newton  Plane Branches} \label{sec : newton-br} 

In this section, we are interested in jet schemes of Newton non-degenerate plane branches. We begin by presenting a well-known combinatorial object: the Newton polygon. First, recall that a plane curve singularity  $\mathcal{C}$ is the zero locus of an (reduced) element $$f=\sum c_{\alpha \beta}x^\alpha y^\beta \in \mathcal{O}_{k^2}$$ where $\mathcal{O}_{k^2}=k[x,y]$ has variables $x=x_1$ and $y=x_2$. More precisely, one can associate to any $f \in \mathcal{O}_{k^2}$ such that $f(0)=0$, its Newton polygon at the origin as follows. Let 
$$\Gamma^+(f,x,y)=\mbox{Conv}\{(\alpha,\beta) + \mathbf{R}_{\geq 0}^2;\:\ c_{\alpha \beta}\not=0\}\subset  \mathbf{R}_{\geq 0}^2$$ where Conv stands for the convex hull, then the Newton polygon $\Gamma(f,x,y)$ is the union of all  compact faces of (the boundary of) $\Gamma^+(f,x,y);$ it has 1-dimensional faces (line segments) and $0-$dimensional faces (vertices). The dual fan of $\Gamma (f,x,y)$ (the Newton dual fan) is a fan (a collection of cones) whose support is $\mathbf{R}_{\geq 0}^2$ and which is defined as follows: take $\omega \in \mathbf{R}^2_{\geq 0};$ and let $$\nu_{\omega}(f)=\min \{ \omega(\alpha,\beta);~~ c_{\alpha \beta}\not=0\},$$ where $\omega(\alpha,\beta)$ can be thought as the scalar product of $\omega$ and $(\alpha,\beta).$ The supporting hyperplane associated with $\omega$ is $$H_{\omega}(f)=\{(\alpha,\beta)\in \mathbf{R}^2  \omega(\alpha,\beta)=\nu_\omega(f)\}.$$ With a face $\Phi$ of $\Gamma(f,x,y),$ one associates the cone $$C_\Phi:=\{\omega \in \mathbf{R}^2; \Phi\subset H_{\omega}(f)\}.$$ The set $\{C_\Phi,\Phi~\mbox{a face of } \Gamma (f,x,y) \}$ is the Newton dual  fan. The tropical variety $T(f)$ of $f$ is the union of the cones $C_\Phi$ associated to each compact segment $\Phi$ of the Newton Polygon of $f$. Let $p$ and $q$ be co-prime integers such that $0<p<q$ (as in Section \ref{stair}). A Newton non-degenerate plane branch (or simply a \emph{Newton plane branch}) $\mathcal{C}$ is the zero locus of an element 
$$f=y^p-cx^q+\sum_{pa+qb>pq}c_{ab}x^ay^b \in \mathcal{O}_{k^2}.$$
(This definition coincides with the one defined  later in Section \ref{sec:newton}).
 In the case of $f$ defining a plane Newton non-degenerate branch, $\Gamma(f,x,y)$ has two vertices $(q,0)$ and $(0,p)$ and one compact face which is the line-segment joining these two vertices. The dual of $(q,0)$ (respectively $(0,p))$ is the 2-dimensional cone generated by $(p,q)$ and $(0,1)$ (respectively  $(1,0)$ and $(p,q))$. Also, the dual cone of the compact segment is the one dimensional cone generated by $(p,q)$. This 1-dimensional cone is in fact $L_{(p,q)}$.

Next, we are going to recall from \cite{M_ts}, the definition of the jet-components graph of a plane curve singularity. 

 Set $\tau_{m+1}=\pi_{m+1,m} : \mathcal{C}_{m+1}\rightarrow \mathcal{C}_m$ the truncation morphism for $m\in \mathbf{N}_0$, and
 denote by $irr(\mathcal{C}^0_m)$ to be the set of irreducible components of the reduced scheme $\mathcal{C}^0_m$. Then, the jet-component graph $Jet(f)$ of $f$ can be defined as follows: 
\begin{itemize}
	\item the vertices of $Jet(f)$ are the elements of $\bigcup_{m\in \mathbf{N}_0}irr(\mathcal{C}^0_m)$,
	\item each vertex is weighted by $(d,e)$ where $d$ is its Krull dimension and $e$ its embedding dimension,
	\item an edge of $Jet(f)$ joins two irreducible components $V_m\in irr(\mathcal{C}_m^0)$ and $V_{m+1}\in irr(\mathcal{C}_{m+1}^0)$ for some $m\in \mathbf{N}_0$ whenever $\tau_{m+1}(V_{m+1})\subseteq V_m.$
\end{itemize}
Finally, an element of $irr(\mathcal{C}^0_m)$ is called a vertex of $Jet(f)$ at the level $m\in \mathbf{N}_0$.
 
It is worth noting that in $Jet(f)$ there is a redundancy of irreducible components of the form $\mathcal{H}^\alpha_m$. In fact, one may have $\mathcal{H}^\alpha_n$ and $\mathcal{H}^\alpha_m$ as components with $n\neq m$. This allows us to construct a new graph where some of these components are removed without losing any information. The next example illustrates these ideas. 
 
 \begin{example}
 \normalfont{Let us draw the jet-components graph of the cusp i.e. $f=y^2-x^3$ in $\mathbf{A}^2$ in the two ways mentioned earlier (See Figure \ref{fig:jet-cusp} below). First, we introduce the notation $(\alpha,w)=\mathcal{H}^\alpha_w$ and we recall that $w=\nu_\alpha(f)-1$.

  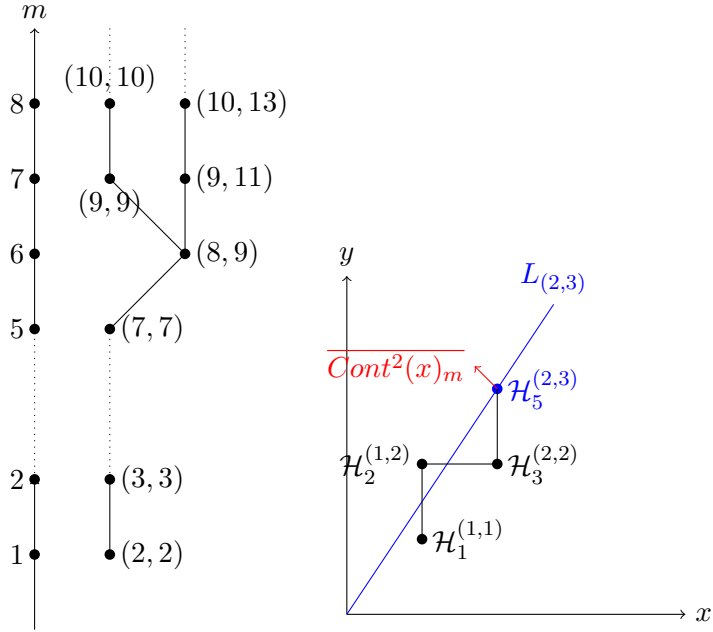
\begin{figure} [ht]
  \centering
  \subfloat{
\begin{tikzpicture}

\fill (0,1) circle[radius=2pt];
\fill (0,2) circle[radius=2pt];
\draw (0, 2) node[left] {$2$};
\fill (0,4) circle[radius=2pt];
\draw (0, 4) node[left] {$5$};

\fill (0,5) circle[radius=2pt];
\draw (0, 5) node[left] {$6$};

\fill (0,6) circle[radius=2pt];
\draw (0, 6) node[left] {$7$};

\fill (0,7) circle[radius=2pt];
\draw (0, 7) node[left] {$8$};

\draw (0, 1) node[left] {$1$};
  \draw[->] (0, 0) -- (0, 2) node[above] {};
  \draw[->] (0, 4) -- (0, 8) node[above] {$m$};
  \draw[dotted] (0, 2) -- (0, 4) node[above] {};
  \draw (1,1) node[right] {$(2,2)$}; 
  \fill (1,1) circle[radius=2pt];
\draw[-] (1,2) -- (1,1);
\draw[dotted] (1,2) -- (1,4);

\draw (1,4) node[right] {$(7,7)$}; 
  \fill (1,4) circle[radius=2pt];

\draw (2,5) node[right] {$(8,9)$}; 
  \fill (2,5) circle[radius=2pt];
\draw[-] (2,5) -- (1,4);

\draw (1,2) node[right] {$(3,3)$}; 
  \fill (1,2) circle[radius=2pt];

\draw (1,6) node[below] {$(9,9)$}; 
  \fill (1,6) circle[radius=2pt];

\draw (2,6) node[right] {$(9,11)$}; 
  \fill (2,6) circle[radius=2pt];

\draw (2,7) node[right] {$(10,13)$}; 
  \fill (2,7) circle[radius=2pt];

\draw (1,7) node[above] {$(10,10)$}; 
  \fill (1,7) circle[radius=2pt];
\draw[-] (1,6) -- (2,5);
\draw[-] (2,5) -- (2,6);

\draw[-] (1,6) -- (1,7);
\draw[-] (2,6) -- (2,7);

\draw[dotted] (1,8) -- (1,7);
\draw[dotted] (2,8) -- (2,7);

  \end{tikzpicture}
  }
    \subfloat{
  \begin{tikzpicture}
\draw[->] (0, 0) -- (4.5, 0) node[right] {$x$};
  \draw[->] (0, 0) -- (0, 4.5) node[above] {$y$};
  \draw[-] [scale=0.5, domain=0:5.5, smooth, variable=\x, blue] plot ({\x}, {1.5*\x}) node[above] {$L_{(2,3)}$};
  \draw[blue] (2,3) node[right] {$\mathcal{H}^{(2,3)}_5$} ; 
  \fill[blue] (2,3) circle[radius=2pt];
  \draw (1,1) node[right] {$\mathcal{H}^{(1,1)}_1$}; 
  \fill (1,1) circle[radius=2pt];
\draw (1,2) node[left] {$\mathcal{H}^{(1,2)}_2$};
\fill (1,2) circle[radius=2pt]; 
\draw (2,2) node[right] {$\mathcal{H}^{(2,2)}_3$}; 
\fill (2,2) circle[radius=2pt];
\draw[-] (1,1) -- (1,2);
\draw[-] (1,2) -- (2,2);
\draw[-] (2,2) -- (2,3);
\draw[->] [red]  (2,3) -- (1.7,3.3);
\draw (1.7,3.3) node[left] {$\textcolor{red}{\overline{Cont^2(x)_m}}$};
  \end{tikzpicture}
   }
   \caption{Different representations of the jet components graph of the cusp.}
  \label{fig:jet-cusp}
\end{figure}

 
 The drawing on the left is obtained by the first way and the one on the right is the new representation of the jet components graph. What makes the drawing on the right more appealing than the one on the left might be the following reasons:
 \begin{itemize}
 	\item The drawing on the right is obtained by an algorithm  whose steps are described in Section \ref{sec:theorem}. It describes a link between the tropical variety and the irreducible components of the jet schemes (which isn't the case for the drawing on the left). In addition, this algorithm is the same for any Newton non-degenerate plane branch (which we will see in the next example).

 	\item When $k=\mathbf{C}$, the graph on the left determines the embedded topological type of the singularity (defined by the cusp in our case) as proven in \cite{M_eq}. Whereas the graph on the right gives richer algebraic data such as an exact description of the $\alpha$ (appearing in the irreducible components of the form $\mathcal{H}^\alpha_m$).
 	
 	\item The weights $(d,e)$ on the left should be added on each vertex in order to know which component vanishes at a certain level (that is when $d=e$) and which does not. However, in the graph on the right, all nonvanishing components lie on the tropical variety.
 	
 	\item The graph on the left can be obtained from the one on the right as follows. For every edge $\alpha\alpha'$ appearing in the graph on the right (where $w=\nu_\alpha(f)-1$ and $w'=\nu_{\alpha'}(f)-1$ are the corresponding weights with $w<w'$), an amount of $w'-w-1$ dots should be added on the edge $\alpha\alpha'$. These dots correspond to $\mathcal{H}^\alpha_m$ for $w\leq  m<w'$.
 \end{itemize}
 
Now, we justify why both representations are essentially the same. According to \cite{M_eq}, every irreducible component $V_m \subset\mathcal{C}^0_m$ can be written as $V_m=\overline{\mathbf{A}^2_m\cap F^\alpha_m}$ for some unique $\alpha\in \mathbf{N}^2$. Using the same notation as in \cite{M_eq}, we note that if $\nu_\alpha(f)>m$ then $V_m=B_m$ and if $\alpha\in L_{(2,3)}$ and  $\nu_\alpha(f)\leq m$, then $V_m=\overline{Cont^{\alpha_1}(x)_m}$ does not vanish and continues to $\infty$. This is due to the fact that $\nu_\alpha(f)=6r$ when $\alpha=(2r,3r)$ for  $r\in \mathbf{N}$.

	}
 \end{example}


 First, we replace $\mathcal{H}^\alpha_m$ by $\alpha$ and remove the text in red (for $f=y^2-x^3$) that  is the name of an irreducible component that does not vanish. This is because  there is a correspondance between the irreducible components of $\mathcal{C}^0_m$ : $m\in \mathbf{N}$ and some weighted points in $\mathbf{R}_{\geq 0}^2$. Here is how it is defined: let $V_m\subset \mathcal{C}^0_m$ be an irreducible component written as $V_m=\overline{\mathbf{A}^2_m\cap F^\alpha_m}$ for a unique $\alpha\in \mathbf{N}^2$. We associate to $V_m$ the weighted point $(\alpha, \infty) $ if $\alpha\in L_{(2,3)}$, $\nu_\alpha(f)\leq m$ and otherwise, we associate to $V_m$ the weighted point $[\alpha, \nu_\alpha(f)-1]$. 

It's worthy to notice that there are two kinds of irreducible components for $\mathcal{C}^0_m$ which can be seen in the next table.
 
\begin{table}[h]
\begin{tabular}{|l|l|l|}
\hline
\textbf{Irreducible components} & \textbf{Interpretation} & \textbf{Jet-components graph} \\ \hline
                      $\mathcal{H}^{\alpha}_m : m<\nu_\alpha(f)$ &    Big, Vanishing & $[\alpha,\nu_\alpha(f)-1]$                                  \\ \hline
                     $\mathbf{A}^2_m\cap F^\alpha_m : \alpha\in T(f)$  & Nonvanishing & $(\alpha,\infty)$                               \\ \hline
                       
\end{tabular}
\caption{Dictionary of notations for $f=y^2-x^3$.} 
\end{table} 
 The weight of the point $\alpha$ is the level $w$ of the jet scheme $\mathcal{C}^0_{w+1}$ where the irreducible component $V_{w+1}$ vanishes: if $w<\infty$,  $V_w=\mathcal{H}^\alpha_w$ is replaced by  either $\mathcal{H}^{\alpha+(0,1)}_{w+1}$ (when $\alpha$ is below $L_{(2,3)}$) or $\mathcal{H}^{\alpha+(1,0)}_{w+1}$ (when $\alpha$ is above $L_{(2,3)}$). If $w=\infty$, then $V_m$ does not vanish.
 
 Concerning $f=y^2-x^3$, we notice that if we remove the weights from all the points $(\alpha,w)\mapsto \alpha$ then these $\alpha$'s follow the same pattern as in Figure \ref{walk(2,3)}. 
\\ \\
Next, we consider the case of a Newton plane branch having a semi-line as a tropical variety.
 
\begin{example} \normalfont{Let $f$ define a Newton non-degenerate plane branch with $L_{(p,q)}$ as a tropical variety and $0<p<q$ coprime with the staircase notation $[[d_1,..., d_p]]$ associated to $(p,q)$ from Theorem \ref{walk}. An inspection shows that Figure \ref{fig:lambda} is the staircase representation $J_{SC}(f)$ associated to $f$} which can be described as follows.
For $\lambda\in L_{(p,q)}\cap \mathbf{N}^2_0$, then
\begin{itemize}
	\item the walker starts at $\lambda+(1,1)$,
	\item for $1\leq i\leq p-1$, the walker moves $d_i$ unit steps vertically followed by one step horizontally,
	\item the walker takes $d_p-1$ steps vertically until reaching $\lambda+(p,q)$,
	\item next, we add a red arrow above $\lambda+(p,q)$,
	\item finally, the walker moves to $\lambda+(p,q)+(1,1)$ and restarts over.
\end{itemize}
 Note that each time the walker moves from a point $\alpha$ to either $\alpha+(0,1)$ or $\alpha+(1,0)$, then $\alpha$ is deleted and  replaced by the new point. In addition,  each point $\alpha$ has a weight of  $\nu_\alpha(f)-1$. Also, the red arrows are never deleted and they are associated to points on the tropical variety with weight $\infty$. 
 In the case where $0<q<p$, vertical steps are replaced by horizontal steps and vice versa.
\end{example}
 
 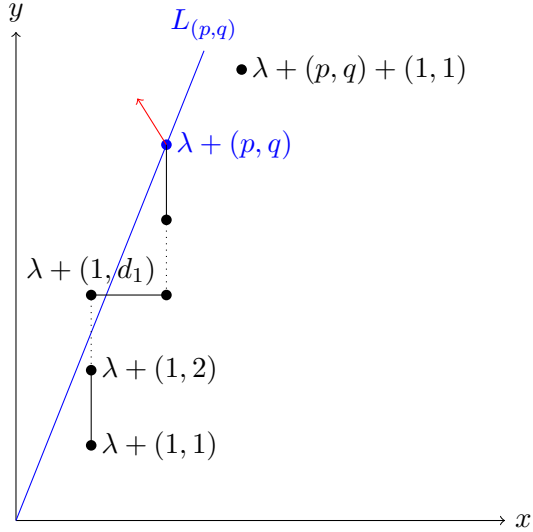
\begin{figure} [ht]
\centering
\begin{tikzpicture}
\draw[->] (0, 0) -- (6.5, 0) node[right] {$x$};
  \draw[->] (0, 0) -- (0, 6.5) node[above] {$y$};
  \draw[-] [scale=0.5, domain=0:5, smooth, variable=\x, blue] plot ({\x}, {5*0.5*\x}) node[above] {$L_{(p,q)}$};
  \draw[blue] (2,5) node[right] {$\lambda+(p,q)$} ; 
  \fill[blue] (2,5) circle[radius=2pt];
  \draw[red, ->] (2,5) -- (2,6,1); 
  \draw (1,1) node[right] {$\lambda+(1,1)$}; 
  \fill (1,1) circle[radius=2pt];
\draw (1,2) node[right] {$\lambda+(1,2)$};
\fill (1,2) circle[radius=2pt]; 
\draw (1,3) node[above] {$\lambda+(1,d_1)$}; 
\fill (1,3) circle[radius=2pt];
\draw[-] (1,1) -- (1,2);
\draw[dotted] (1,2) -- (1,3);
\draw[-] (1,3) -- (2,3);
\draw[-] (2,4) -- (2,5);
\draw[dotted] (2,3) -- (2,4);
\fill (2,3) circle[radius=2pt];
\fill (2,4) circle[radius=2pt];
\fill (3,6) circle[radius=2pt];
\draw (3,6) node[right] {$\lambda+(p,q)+(1,1)$}; 
  \end{tikzpicture}
  \caption{Drawing of the unweighted points representing the irreducible components of $\mathcal{C}^0_m$ : $m\in \mathbf{N}$. The red arrow corresponds to an irreducible component that continues to infinity.}
  \label{fig:lambda}
\end{figure}
 
 The next proposition shows that $J_{SC}(f)$ determines $Jet(f)$.
 
 \begin{proposition}
 	If $f,g$ define some Newton plane branches with $J_{SC}(f)=J_{SC}(g)$ then we have $Jet(f)=Jet(g)$.
 \end{proposition}
 \begin{proof}
 Consider first an edge in Figure \ref{fig:lambda}. Denote by $(\alpha,w)$ and $(\alpha',w')$ with $w<w'<\infty $ the weighted vertices of this edge. Note that $\alpha'=\alpha+(1,0)$ or $\alpha'=\alpha+(0,1)$ with $w=\nu_\alpha(f)-1$ and $w'=\nu_{\alpha'}(f)-1$. Add on the edge $\alpha\alpha'$, $w'-w-1$ vertices representing the big components defined by the same $\alpha$. 
 Now consider a nonvanishing component represented by $(\alpha, \infty)$. Add on the arrow above $\alpha$ a countably infinite amount of vertices. These represent the nonvanishing components for $m\geq \nu_\alpha(f)$ and $\alpha\in T(f)$. Finally, scale each weight $w$ by its level $m=w+1$.\\
This construction that we just described allows us to obtain $Jet(h)$ from $J_{SC}(h)$ for every polynomial (or power series) $h$ which is non-degenerate (with respect to $x$ and $y$). One of the advantages of this construction is that it depends only on the weighted graph $J_{SC}(h)$ and not on the polynomial $h$.
 This allows us to conclude that since $J_{SC}(f)=J_{SC}(g)$ then $Jet(f)=Jet(g)$. \end{proof}
\section{Jet Schemes of Newton Plane Curves} \label{sec:newton} 

We call $\mathcal{C}\subset k^2$ a \emph{Newton non-degenerate plane curve} (or simply a \emph{Newton plane curve}) if it is the zero locus of some reduced 
$$f=\sum_{(\alpha,\beta)\in supp(f)} c_{\alpha,\beta}\cdot x^\alpha y^\beta\in \mathcal{O}_{k^2}$$
 where for every face $\Phi$ of the Newton Polygon of $f$,  the restriction of $f$ to $\Phi$
 $$f_{\Phi}=\sum_{(\alpha,\beta)\in \Phi\cap supp(f)}c_{\alpha,\beta}\cdot x^\alpha y^\beta$$ 
has no critical points in the algebraic torus $(k^\times)^2$  i.e. the system of equations:
$$f_\Phi = \frac{\partial f_\Phi}{\partial x}=\frac{\partial f_\Phi}{\partial y}=0$$
has no solutions in this torus.

It is worth mentioning that it is crucial in our study  to know the description of Newton plane curves in order to compute their jet schemes over the origin of $k^2$. In this section, we will describe the equations of these curves.

First, we start with Proposition \ref{th : new-br} that describes the equation of a Newton plane branch up to an isomorphism. 

\begin{proposition} \cite{zbMATH05577991}\label{th : new-br}
The Newton plane branches are up to an isomorphism the plane curves defined by

$$f = y^p - cx^q + \sum_{pa+qb>pq} c_{a,b} x^ay^b \in \mathcal{O}_{k^2}$$
with $p,q\geq 0$ coprime integers (and $f$ has only one Puiseux characteristic pair).
\end{proposition}

 Recall that every curve which is contained in a Newton plane curve is also Newton non-degenerate. 
So we can describe the equations of Newton plane curves through their defining polynomials (or power series) as well as the (equational) relations between their different branches.

\begin{theorem} \label{thm: newton shape} 
Any Newton plane curve $\mathcal{C}\subset k^2$ is defined (up to an isomorphism) by:

\begin{enumerate}
\item each branch $\mathcal{B} \subset \mathcal{C}$ is given by some $g\in \mathcal{O}_{k^2}$ such that 

$$g = y^p - cx^q + \sum_{pa+qb>pq} c_{a,b} x^ay^b.$$

\item If $g' = y^{p'} - c' x^{q'} + \sum_{p'a+q'b>p'q'} c'_{a,b} x^ay^b$ is another branch of $\mathcal{C}$ with $(p,q) = (p',q'),$ then $c\neq c'$.

\end{enumerate}
\end{theorem}

In order to prove Theorem \ref{thm: newton shape}, we recall from \cite{zbMATH06142113} an equivalent definition for Newton non-degenerate polynomials which can be described as follows. Let $f=\sum_{a,b}f_{a,b}\cdot x^ay^b\in k[x,y]$ be reduced and $w\in \mathbf{R}_{\geq 0}^2$, the initial form 
$$in_w(f) = \sum_{(a,b)\cdot w = \nu_w(f)}f_{a,b}\cdot x^ay^b$$
where $(a,b)\cdot w = aw_1+bw_2$ is the natural inner product. Then, $f$ defines a Newton Plane Curve if and only if for all $w\in \mathbf{R}_{\geq 0}^2$, the reduced algebraic set defined by $in_w(f)$ has no singularities in the torus $(k^\times)^2$.

Equivalently, $f$ is Newton non-degenerate if and only if for all coprime  $(p,q)\in T(f)\cap \mathbf{N}^2$, then $in_{(p,q)}(f)$  is the product of a non-zero constant by a monomial in $x,y$ (which might be equal to $1$) and a finite product of binomials in $x,y$:
$\prod_{j}(y^p-c_j x^q)$
where $j \mapsto c_j$ is injective.

Now, that we know the equations of the branches of a Newton plane curve and how they are related, let us consider the following example.

\begin{example}
The tropical variety of $f=(y^2-x^3)(y^3-x^2)$ is $T(f)=L_{(2,3)}\cup L_{(3,2)}$.
Denote by $\sigma=L_{(2,3)}+L_{(3,2)}$ the cone whose boundary is $T(f)$.
	The special staircase walk for $f=(y^2-x^3)(y^3-x^2)$ determines the irreducible components of $\mathcal{C}^0_m$ for  $m\in \mathbf{N}$.

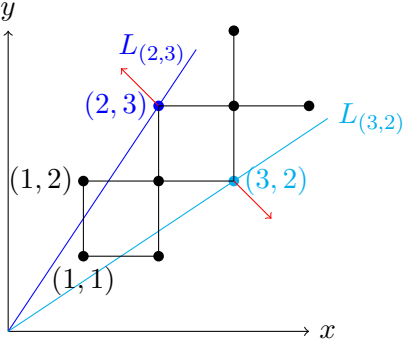
\begin{figure}[h] \centering
\begin{tikzpicture}
\draw[->] (0, 0) -- (4, 0) node[right] {$x$};
  \draw[->] (0, 0) -- (0, 4) node[above] {$y$};
  \draw[-] [scale=0.5, domain=0:5, smooth, variable=\x, blue] plot ({\x}, {1.5*\x}) node[left] {$L_{(2,3)}$};
  \draw[-] [scale=0.5, domain=0:8.5, smooth, variable=\x, cyan] plot ({\x}, {2/3*\x}) node[right] {$L_{(3,2)}$};
  \draw[blue] (2,3) node[left] {$(2,3)$} ; 
  \fill[blue] (2,3) circle[radius=2pt];
  
  \draw[cyan] (3,2) node[right] {$(3,2)$} ; 
  \fill[cyan] (3,2) circle[radius=2pt];
  
  \draw (1,1) node[below] {$(1,1)$}; 
  \fill (1,1) circle[radius=2pt];
\draw (1,2) node[left] {$(1,2)$};
\fill (1,2) circle[radius=2pt]; 

\fill (2,2) circle[radius=2pt];
\fill (2,1) circle[radius=2pt];
\draw[-] (1,1) -- (1,2);
\draw[-] (1,2) -- (2,2);
\draw[-] (2,2) -- (2,3);

\draw[-] (1,1) -- (2,1);
\draw[-] (2,1) -- (2,2);
\draw[-] (2,2) -- (3,2);
\draw[red, ->] (2,3) -- (2-0.5,3+0.5);
\draw[red, ->] (3,2) -- (3+0.5,2-0.5);

\fill (3,3) circle[radius=2pt]; 
\fill (3,4) circle[radius=2pt]; 
\fill (4,3) circle[radius=2pt]; 
\draw[-] (2,3) -- (3,3);
\draw[-] (3,2) -- (3,3);
\draw[-] (3,3) -- (3,4);
\draw[-] (3,3) -- (4,3);
\end{tikzpicture}
  \caption{Special staircase for $f=(y^2-x^3)(y^3-x^2)$} \label{walk:T(f)}
\end{figure}
\end{example}
In fact, every irreducible component $V_m\subset\mathcal{C}^0_m$ is written as $V_m=\overline{\mathbf{A}^2_m\cap F^\alpha_m}$ for some unique $\alpha\in \mathbf{N}^2$. Then we associate with $V_m$  the weighted point $(\alpha, \infty)$ if $\alpha\in T(f)$, $\nu_\alpha(f)\leq m$ and otherwise, we associate with $V_m$ the weighted point $[\alpha,\nu_\alpha(f)-1]$.
Thus, we obtain Figure \ref{walk:T(f)} from the following rules:

\begin{itemize}
\item The walker starts always at $(1,1)$,
\item if $\alpha\in \overset{\circ}{\sigma}$, then we replace $\alpha$  with $\alpha+(0,1)$ and $\alpha+(1,0)$,
\item if $\alpha\in L_{(2,3)}$, then we draw a red arrow above $\alpha$ and add $\alpha+(1,0)$,
\item if $\alpha$ is above $L_{(2,3)}$ then we replace $\alpha$ by $\alpha+(1,0)$,
\item if $\alpha\in L_{(3,2)}$, then we draw a red arrow above $\alpha$ and add $\alpha+(0,1)$,
\item if $\alpha$ is below $L_{(3,2)}$ then we replace $\alpha$ by $\alpha+(0,1)$.
\end{itemize}

The weighted point $(\alpha,w)$ represents the irreducible components $V_m=\overline{\mathbf{A}^2_m\cap F^\alpha_m}$  that "vanish" at the level $w+1$ for $m<w<\infty$. For $w=\infty$, then $V_m$ continues to $w=\infty$ without vanishing. Those $V_m$ that do not vanish, are represented by a red arrow. 

Next, we define the order of a power series $h=\sum_{k} h_k t^k \in k[[t]]\backslash \{0\}$. First, we set $ord_t(h)$ to be the least exponent $e$ such that $h_e \neq 0$. Otherwise if $h=0$, we set 
$$ord_t(0) = \infty.$$

Recall, now, from \cite{CDG} the definition of the semigroup of a reduced plane curve with $r$ branches. Let $f=f_1...f_r$ be a reduced polynomial defining such a curve with $k=\mathbf{C}$. Choose $\phi_i : \mathbf{C}_t\rightarrow \mathbf{C}_{x,y}^2$ a parametrisation (with variable $t$) of $\{f_i(x,y)=0\}$, obtained by the Newton-Puiseux theorem. Denote by $v_i(g)=(f_i,g)_0=ord_t(g\circ \phi_i)$ for $g\in \mathbf{C}[[x,y]]$.

The \emph{semigroup} of $f$ is the subsemigroup $\mathcal{S}(f)\subseteq \mathbf{N}^r_0$ whose elements are given by:
$$v(g)=[v_1(g),..., v_r(g)]$$
with $v_i(g)<\infty$ for all $g\in \mathbf{C}[x,y]$ and index $i$.

It is known \cite{waldi1} that the semigroup determines the embedded topological type of the singularity defined by $f$. When $f$ is Newton non-degenerate, we can rewrite this fact as follows.
\begin{theorem} \label{th : semigrp}
	Let $f=f_1...f_r\in \mathcal{O}_{k^2}$ be a Newton non-degenerate product of $r$ branches and let $g=g_1...g_r$ be another such polynomial with the same number of branches. Assume
	$$T(f_i)=T(g_i)$$
    for $1\leq i\leq r$ 
	and for $j\neq i$,
	$$(f_i,f_j)_0 = (g_i,g_j)_0$$
are equal. Then $f$ and $g$ have the same embedded topological type of the singularity at $0$.
\end{theorem}

\section{Useful Results} \label{sec:crucial}

In this part of the article, we shall present the key results that are very useful in proving our theorems in the remaining sections.

We say that two polynomials $f,g\in k[x_1,..., x_n]$ are \textit{congruent modulo a variety} $W\subset \mathbf{A}^\infty$ (which we write $f=g \mod W$) if they are congruent modulo the defining ideal of $W$ in $k[x_1,x_2, ...]$.

Now for $\alpha\in \mathbf{R}^2$, we set $|\alpha|=\alpha_1+\alpha_2$. With this in mind,  we can introduce the terminology, $f=g \mod \alpha$ when $W=\mathcal{H}^\alpha_\infty$.  But first we need the following observation which will be useful for later.
\begin{observation} \label{obs:branch}
	Let $f\in \mathcal{O}_{k^2}$ define a non-degenerate Newton plane branch. Let $\alpha = s(p,q)\in T(f)$ with $gcd(p,q)=1$. Then, $\nu_\alpha(f)>|\alpha|$ if $p,q>1$.
\end{observation}
\begin{proof}
	Write $in_\alpha(f)=y^p-cx^q$. Then, $pq>p+q$ since $p,q>1$. 
	$$\nu_\alpha(f)=spq>s(p+q)=|\alpha|.$$
This proves our observation.
\end{proof} 

The next lemma will be used (in the next section) in proving Theorem \ref{th : decom}.
\begin{lemma} \label{lm : T}
	Let $f\in \mathcal{O}_{k^2}$ define a non-degenerate plane curve with $T(f)\simeq \mathbf{R}_{\geq 0}$. Set $f=f_1...f_r$ to be the decomposition of $f$ into branches. Let $\alpha\in T(f)\cap \mathbf{N}^2$ and $m\geq \nu_\alpha(f)$. Then we have
	$$F^\alpha_{m}=\bigsqcup_{i=1}^r F^\alpha_{i,\nu_\alpha(f_i)+m-\nu_\alpha(f)}.$$
\end{lemma}
\begin{proof}
By induction on $m\geq \nu_\alpha(f)$.
\begin{itemize}
	\item At $m=\nu_\alpha(f)$, we have the following:
	$$f^{[\nu_\alpha(f)]}=f_1^{[\nu_\alpha(f_1)]} ... f_r^{[\nu_\alpha(f_r)]} \mod \alpha$$
	and since $T(f)$ is the semi-line that passes through $(0,0)$ and $\alpha$, then for all $i$ it holds that
	$$f_i^{[\nu_\alpha(f_i)]} = in_{\alpha}(f_i)(x^{(\alpha_1)}, y^{\alpha_2}) \mod \alpha$$
	which is clearly not a monomial.
	Moreover, the $in_\alpha(f_j)$ are pairwise coprime as polynomials since $f$ is Newton non-degenerate. This allows us to conclude that:
	$$F^\alpha_{\nu_\alpha(f)} = \bigsqcup_{i=1}^r F^\alpha_{i,\nu_\alpha(f_i)}, $$
	which proves the initial case $m=\nu_\alpha(f)$.
	\item Suppose that for $m-1$, we have:
	$$F^\alpha_{m-1}=\bigsqcup_{i=1}^r F^\alpha_{i,\nu_\alpha(f_i)+m-1-\nu_{\alpha}(f)}.$$
	Then, as $F^\alpha_m = F^\alpha_{m-1} \cap \{f^{(m)}=0\}$, we need to show that:
	$$F^\alpha_{i,\nu_\alpha(f_i)+m-1-\nu_\alpha(f)}\cap \{f^{(m)}=0\} = F^\alpha_{i,\nu_\alpha(f_i)+m-\nu_\alpha(f)} .$$
	However, this can be seen from the fact that if $k=\nu_\alpha(f_i)+m-\nu_\alpha(f)$ then
	$$f^{(m)} = f^{(k)}_i \prod_{j\neq i} f^{[\nu_\alpha(f_i)]}_i \mod F^\alpha_{i,k-1}.$$
	So the proposition is true for $m$, and this concludes the induction.	
\end{itemize}

\end{proof}

In the next section, we shall show that the irreducible components of the jet schemes $\mathcal{C}^0_m$ (over the origin)  of Newton plane curve singularities $\mathcal{C}=\{f=0\}\subset \mathbf{A}^2$ are of two kinds:
\begin{itemize}
	\item the ones of the form $\mathcal{H}^\alpha_m$ that are irreducible and having $|\alpha|$ as codimension. In this case, the $\alpha$'s are chosen so  that there is no $\beta\leq \alpha$ component-wise with $\alpha\neq\beta$. These $\mathcal{H}^\alpha_m$ are said to have hyperplane coordinates;
	\item the ones of the form $\overline{G^\alpha_k\cap \mathbf{A}^2_m}$ where  $\{g=0\} \subset \mathcal{C}$ is a branch, $m\geq \nu_\alpha(f)$, $k=\nu_\alpha(g)+m-\nu_\alpha(f)$ and $\alpha\in T(g)\cap \mathbf{N}^2$; these are called infinite components.
\end{itemize}

The next lemma is concerned with determining  their codimensions and verifying their irreducibility.
\begin{lemma} \label{lm : codim+irr}
	Let $f\in \mathcal{O}_{k^2}$ define a Newton plane branch. Let $\alpha\in T(f)\cap \mathbf{N}^2$ and $\nu_\alpha(f)\leq k\leq m$. Then, we have $$codim(F^\alpha_k\cap \mathbf{A}^2_m ; \mathbf{A}^2_m)=|\alpha|+k-\nu_\alpha(f)+1$$ 
	and $F^\alpha_k\cap \mathbf{A}^2_m$ is irreducible.
\end{lemma}
\begin{proof}
	
	Denote by $\pi : \mathbf{A}^2_m \rightarrow \mathbf{A}^2_k$ the truncation morphism of the jet schemes of $\mathbf{A}^2$. They are defined as $\pi(x_0,..., x_m,y_0,..., y_m) = (x_0,..., x_k,y_0,..., y_k)$ where $v_i = v^{(i)}$ if $v=x$ or $v=y$. Notice that
	$$F^\alpha_k\cap \mathbf{A}^2_m = \pi^{-1}(F^\alpha_k\cap \mathbf{A}^2_k).$$
	In order to have $F^\alpha_k\cap \mathbf{A}^2_m$ irreducible, it's enough to have $F^\alpha_k\cap \mathbf{A}^2_k$ irreducible. Since $f$ defines a Newton plane branch, then $F^\alpha_k\cap \mathbf{A}^2_k$ is irreducible by Lemma 4.3 of \cite{M_eq} and using our notation it has codimension
	$|\alpha|+k-\nu_\alpha(f)+1$. Note that at the level $m$, the irreducibility of $F_k^\alpha\cap \mathbf{A}^2_m$ is preserved as well as the codimension since it has the same defining equations as $F^\alpha_k\cap \mathbf{A}^2_k$.
\end{proof}

To distinguish the candidates for the irreducible components of $\mathcal{C}^0_m$ for all $m$, we only need to compare them (by inclusion) as their irreducibility has been proven in the previous result.
\\ \\
First, let us remind the reader with our notation:  $\mathbf{A}^2_m$ is the $m$-jet space of $\mathbf{A}^2$ with variables $x^{(0)},..., x^{(m)},y^{(0)},..., y^{(m)}$ and for $\alpha\in \mathbf{N}^2$, 
\begin{itemize}
	\item $\mathcal{H}^\alpha_m = \{x^{(i)}=y^{(j)} = 0 : 0\leq i<\alpha_1, 0\leq j<\alpha_2\}\subset \mathbf{A}^2_m$;
	\item Let $g$ define a branch in $\mathbf{A}^2$ and $\nu_\alpha(g)\leq k\leq m$, Next, set
	$$\mathbf{A}^2_m\cap G^\alpha_k = \{x^{(\alpha_1)}\neq 0\}\cap \{y^{(\alpha_2)}\neq 0\}\cap \mathcal{H}^\alpha_m\cap \{g^{[\nu_\alpha(g)]}=...=g^{(k)}=0\}.$$
\end{itemize}

With this in mind, we can sum up the purpose of this section in Table 2 which deals with how to determine the irreducible components of $\mathcal{C}^0_m$. Whenever $\mathcal{H}^\gamma_m$ is mentioned we impose the condition $\nu_\gamma(f)>m$ and whenever 
$$\mathbf{A}^2_m\cap F^\alpha_{i,a}$$ 
corresponds to a branch defined by a factor $f_i$ of $f$ with $\alpha\in T(f_i)$, we impose the condition $m\geq \nu_\alpha(f)$ and  $a=\nu_\alpha(f_i)+m-\nu_\alpha(f)$.


\begin{table}[ht] 

\begin{tabular}{|l|l|l|l|}
\hline
$A$ & $B$ & when does $A\subseteq B?$ & reason behind $A\subseteq B$\\ \hline
                      $\mathcal{H}^{\alpha}_m$ &    $\mathcal{H}^\beta_m$ & $\alpha\geq \beta$ component-wise &                                   \\ \hline
                     $\mathcal{H}^{\gamma}_m$ &    $\overline{\mathbf{A}^2_m\cap F^\alpha_{i,a}}$ & $\gamma=\alpha+(1,1)$, $m=\nu_\alpha(f)$ &                                         this holds when $T(f) = T(f_i)$                       \\ \hline
                     $\overline{\mathbf{A}^2_m\cap F^\alpha_{i,a}}$ &  $\mathcal{H}^\gamma_m$   & never & otherwise, $\alpha\geq\gamma$ absurd                                                              \\ \hline
                      $\overline{\mathbf{A}^2_m\cap F^\alpha_{i,a}}$ &  $\overline{\mathbf{A}^2_m\cap F^\beta_{j,b}}$   & $\alpha=\beta$ and $ f_i = f_j$ & compare codimensions                                                             \\ \hline
                   
\end{tabular}
    
\caption{Table describing when the candidates $A,B \subseteq \mathcal{C}^0_m$ verify $A\subseteq B$.} 
\end{table} 

\begin{lemma} \label{lm : minimal}
    Let $f\in \mathcal{O}_{k^2}$ define a Newton plane curve singularity, $\alpha\in T(f)$, with $\nu_\alpha(f)\leq m$. Then, there is no $(1,1)\leq \lambda<\alpha$ with $\nu_\lambda(f)>m-1$ i.e. $\mathcal{H}^\lambda_{m-1}\subseteq \mathcal{C}^0_{m-1}$.
\end{lemma}
\begin{proof}
    Suppose such  $\lambda$ exists. Then, $\nu_\lambda(f)\geq m\geq \nu_\alpha(f)$. However, as $\lambda<\alpha$ then $\nu_\lambda(f)\leq\nu_\alpha(f)$ and $m=\nu_\alpha(f) = \nu_\lambda(f)$. Let's set $H=\mathcal{H}^{\lambda+\varepsilon}_m$ and $V_i = \mathcal{H}^\alpha_m\cap \{f^{[\nu_\alpha(f_i)]}_i=0\}$ where $f_i$ is an irreducible factor of $f$ with $\alpha\in T(f_i)$, and
    $$\mathcal{H}^\alpha_{m-1}\subsetneq\mathcal{H}^\lambda_{m-1}\subseteq \mathcal{C}^0_{m-1},$$
    where $\varepsilon\in \{(1,0),(0,1)\}$ and $\pi_{m,m-1}:\mathcal{C}^0_m\rightarrow \mathcal{C}^0_{m-1}$ is the truncation morphism. By considering the preimage $\pi^{-1}_{m,m-1}$, we obtain $$V_i\subsetneq \mathcal{H}^\alpha_m\cap \{f^{(m)}=0\}\subseteq H.$$ 
   Since $H$ is irreducible and $V_i$ does not contain $H\backslash V_i$, then $H=\overline{H\backslash V_i}$ and
   $$(x^{(0)},..., x^{(\lambda_1-1)},y^{(0)},..., y^{(\lambda_2-1)})=\mathcal{I}_{H}=\{r\in \mathcal{O}_{\mathbf{A}^2_m} : r\cdot \mathcal{I}_{V_i}\subseteq \mathcal{I}_H\}$$ 
   is an equality for the defining ideals $\mathcal{I}_H = (\mathcal{I}_H : \mathcal{I}_{V_i})$ (colon ideal). By denoting  $r=x^{(\lambda_1-1)}$ and $g=in_\alpha(f_i)(x^{(\alpha_1)},y^{(\alpha_2)})\in \mathcal{I}_{V_i}$ in
   $r\cdot g\in \mathcal{I}_{H}$, then we can write $r\cdot g = \sum_{i=0}^{\alpha_1-1}x^{(i)}X_i+\sum_{j=0}^{\alpha_2-1}y^{(j)}Y_j$. Moreover, by annihilating the variables $x^{(0)},..., x^{(\lambda_1-2)}, y^{(0)},..., y^{(\alpha_2-1)}$, we get
   $$x^{(\lambda_1-1)}[{y^{(\alpha_2)}}^p-c_i\cdot {x^{(\alpha_1)}}^q]=\sum_{i=\lambda_1-1}^{\alpha_1-1}x^{(i)}X'_i.$$
   Now, we distinguish the following 2 cases.
   \begin{itemize}
       \item If $\alpha_1>\lambda_1$ then by annihilating $x^{(\alpha_1)},x^{(\alpha_1+1)},..., x^{(m)},y^{(\alpha_2+1)},..., y^{(m)}$, we notice that
   $$x^{(\lambda_1-1)}{y^{(\alpha_2)}}^p=\sum_{i=\lambda_1-1}^{\alpha_1-1}x^{(i)}X''_i.$$
   So annihilating the variables
   $x^{(\lambda_1)},..., x^{(\alpha_1-1)}$ leads to a contradiction because the non-zero monomial $x^{(\lambda_1-1)}{y^{(\alpha_2)}}^p=0$ in a power series ring.
       \item If $\alpha_1=\lambda_1$ then $\alpha_2>\lambda_2$.  By taking $r'=y^{(\lambda_2-1
       )}$ and $g'=g$,  we can write
       $$r'g=y^{(\lambda_2-1)}\cdot [{y^{(\alpha_2)}}^p-c_i{x^{(\alpha_1)}}^q]=\sum_{i=0}^{\alpha_1-1}x^{(i)}U_i+\sum_{j=0}^{\alpha_2-1}y^{(j)}V_j,$$
       which after annihilating the variables
       $$x^{(0)},..., x^{(\alpha_1-1)},y^{(j)} : j\neq \lambda_2-1$$
       leads to  $-c_i\cdot y^{(\lambda_2-1)}{x^{(\alpha_1)}}^q=0$ (in a power series ring)  which  can not be the case.
   \end{itemize}
   Therefore, in both cases we arrived at a contradiction. Thus, the proof is complete.
\end{proof}

\begin{theorem} \label{th : M1+}
Let $f\in \mathcal{O}_{k^2}$ define a Newton plane curve singularity such that $T(f)=T(f_i)$ where $f_i$ defines a branch of $f$. 
	Let $\alpha,\beta\in T(f)\cap \mathbf{N}^2$ and $\gamma= \beta+(1,1)$ and suppose $\nu_\gamma(f)>m\geq \nu_\alpha(f)$. Then, we have 
	$$\mathcal{H}^{\gamma}_m \subseteq \overline{F^\alpha_{i,\nu_\alpha(f_i)+m-\nu_\alpha(f)}\cap \mathbf{A}^2_m}$$
	if and only if $\alpha=\beta$ and $m=\nu_\alpha(f)$. Moreover, the inclusion on the left is always strict. \end{theorem}

\begin{proof}
Suppose first that $\alpha=\beta$ and that
$\mathcal{H}^{\gamma}_m =\overline{F^\alpha_{i,\nu_\alpha(f_i)+m-\nu_\alpha(f)}\cap \mathbf{A}^2_m}$. Then, 
$$U := F^\alpha_{i,\nu_\alpha(f_i)+m-\nu_\alpha(f)}\cap \mathbf{A}^2_m \subset \mathcal{H}^{\gamma}_m.$$
However, $x^{(\alpha_1)}\neq 0$ in $U$ and $x^{(\alpha_1)}=0$ in $\mathcal{H}^\gamma_m$, which is a contradiction. So the inclusion if it occurs is always strict. 
Now, let us show the equivalence in the theorem. Suppose that $m=\nu_\alpha(f)$ and $\alpha=\beta$, then 
$$\overline{F^\alpha_{i,\nu_\alpha(f_i)}\cap \mathbf{A}^2_m}=\{x^{(0)}=...=x^{(\alpha_1-1)}=y^{(0)}=...=y^{(\alpha_2-1)}=f_i^{[\nu_\alpha(f_i)]}=0\}$$
with $f^{[\nu_\alpha(f_i)]}_i=[{y^{(\alpha_2)}}]^p-c_i\cdot [{x^{(\alpha_1)}}]^q\mod \alpha$ where $T(f_i)=\mathbf{R}_{\geq 0}(p,q)$. So that we have the inclusion $\mathcal{H}^{\gamma}_m \subseteq \overline{F^\alpha_{i,\nu_\alpha(f_i)+m-\nu_\alpha(f)}\cap \mathbf{A}^2_m}$, with $m=\nu_\alpha(f)$ and $\gamma=\alpha+(1,1)$.

To prove the other direction of the equivalence, suppose that we have the inclusion above with $m\geq \nu_\alpha(f)$ together with the condition $\alpha=\beta$. Then, 
$$\mathcal{H}^{\gamma}_m \neq  \overline{F^\alpha_{i,\nu_\alpha(f_i)+m-\nu_\alpha(f)}\cap \mathbf{A}^2_m}.$$
By comparing the codimensions of these irreducible sets: 
$$|\gamma|=|\alpha|+2 > |\alpha|+1+m-\nu_\alpha(f),$$
we arrive at the inequality $1 > m-\nu_\alpha(f)\geq 0$ which shows that $m=\nu_\alpha(f)$. 

Now, assume on the other hand that $\alpha\neq \beta$ with $\mathcal{H}^\gamma_m\subseteq \overline{F^\alpha_{i,\nu_\alpha(f_i)+m-\nu_\alpha(f)}\cap \mathbf{A}^2_m}$. Let us recall that for every reduced plane curve $\mathcal{C}=\{g=0\}\subset \mathbf{A}^2$,
$$m<\nu_\gamma(g)\Leftrightarrow (x^{(0)},y^{(0)},g^{(j)})_{0\leq j\leq m}\subseteq(x^{(i)},y^{(j)})_{0\leq i< \gamma_1\:\ 0\leq j<\gamma_2}.$$
As $\mathcal{H}^\gamma_m\subseteq \mathcal{H}^\alpha_m\cap \{f^{(j)}_i=0\}_{0\leq j\leq a}\subseteq \mathcal{C}^0_{i,a}$, then
$$(x^{(0)},y^{(0)},f^{(j)}_i)_{0\leq j\leq a}\subseteq (x^{(i)},y^{(j)})_{0\leq i< \gamma_1\:\ 0\leq j<\gamma_2}$$
and $\nu_\gamma(f_i)>a=\nu_\alpha(f_i)+m-\nu_\alpha(f)$. If $a+h=\nu_\gamma(f_i)-1$ then
$$(x^{(0)},..., x^{(\alpha_1-1)},y^{(0)},..., y^{(\alpha_2-1)}, f_i^{(j)})_{0\leq j\leq a+h} \subseteq (x^{(i)},y^{(j)})_{0\leq i<\gamma_1\:\ 0\leq j<\gamma_2},$$
$$\mathcal{H}^\gamma_{a+h} \subset \mathcal{H}^\alpha_{a+h}\cap \{f_i^{(j)}=0\}_{\nu_\alpha(f_i)\leq j\leq a+h}$$
is a strict inclusion. Therefore, $|\gamma|\geq 1+a+h-\nu_\alpha(f_i)+1+|\alpha|$. By replacing $\gamma=\beta+(1,1)=b\cdot \alpha+(1,1)$, we find that
$$b|\alpha|+2=|\beta|+2 = |\gamma|\geq 2+|\alpha|+\nu_\gamma(f_i)-1-\nu_\alpha(f_i),$$
$$b|\alpha|\ge |\alpha|+\nu_\gamma(f_i)-1-\nu_\alpha(f_i).$$
As $\nu_\gamma(f_i)>a\geq \nu_\alpha(f_i)$ then $b|\alpha|\geq |\alpha|$ and $b\geq 1$. We have $\mathcal{H}^\beta_{m-1}\subseteq \mathcal{H}^\alpha_{m-1}$ i.e. $\beta\geq \alpha$. Moreover, by the previous lemma, there is no $\lambda<\beta$ such that $\nu_\lambda(f)>m-1$. Thus, we must have $\alpha=\beta$.

Now suppose $m=\nu_\alpha(f)$. If
$$\mathcal{H}^\gamma_m \subseteq \overline{F^\alpha_{i,\nu_\alpha(f_i)}\cap \mathbf{A}^2_m}= \mathcal{H}^\alpha_m\cap \{f_i^{[\nu_\alpha(f_i)]}=0\}$$
then again by comparing codimensions, (since the inclusion is strict), we get $|\gamma| >1+|\alpha|$. Writing $\beta=b\cdot \alpha$ (with $b>0$) in
$$(b-1)|\alpha|=|\beta|-|\alpha|=|\gamma|-(2+|\alpha|)\geq 0,$$
guarantees that  $b\geq 1$ so that $\mathcal{H}^\beta_{m-1}\subseteq \mathcal{H}^\alpha_{m-1}$ i.e. $\beta\geq \alpha$. As there is no $\lambda<\beta$ such that $\nu_\lambda(f)>m-1$, we again must have $\alpha=\beta$.
\end{proof}

Our next theorem deals with the fact that distinct infinite components (of the jet schemes of a Newton Plane Curves) come from distinct branches of this curve. This in turns means that
whenever one such irreducible component is inside another infinite component, then those components are essentially the same.

Before doing so, we need to recall the formula of the codimension of $G^\alpha_k\cap \mathbf{A}^2_m$ (where $G$ is defined by an irreducible Newton non-degenerate $g$ that divides $f$) is given by:
$$codim(G^\alpha_k\cap \mathbf{A}^2_m ; \mathbf{A}^2_m)=1+|\alpha|+m-\nu_\alpha(f)$$
in which we set $k=\nu_\alpha(g)+m-\nu_\alpha(f)$.
With this in mind, we have the following theorem.
\begin{theorem}\label{lm : inclusion}
	Let $f\in \mathcal{O}_{k^2}$ define a Newton plane curve singularity. Let $f_i$ and $f_j$ define branches of $f$ such that 	$$T(f_l)=\mathbf{R}_{\geq 0} \cdot \alpha_l \mbox{ with } \alpha_l\in \mathbf{N}^2$$ 
	and $\nu_{\alpha_l}(f)\leq m$ for  $l\in\{i,j\}$. Set $k_l=\nu_{\alpha_l}(f_l)+m-\nu_{\alpha_l}(f)$. If 
	$$F^{\alpha_i}_{i,k_i}\cap \mathbf{A}^2_m \subseteq \overline{F^{\alpha_j}_{j, k_j}\cap \mathbf{A}^2_m}$$ 
	then $\alpha_i=\alpha_j$ and 
	$$\overline{\mathbf{A}^2_m\cap F^{\alpha_{i}}_{i,k_i}}=\overline{\mathbf{A}^2_m\cap F^{\alpha_{j}}_{j,k_j}}.$$
	In particular, $f_i=f_j$.
\end{theorem}
\begin{proof}
	The inclusion: $F^{\alpha_i}_{i,k_i}\cap \mathbf{A}^2_m \subseteq \overline{F^{\alpha_j}_{j, k_j}\cap \mathbf{A}^2_m} $ tells us (via comparing codimensions) that
	$$1+m-\nu_{\alpha_i}(f)+|\alpha_i|\geq 1+m-\nu_{\alpha_j}(f)+|\alpha_j|,$$ which in turn leads to the following inequality
	\begin{equation}
		|\alpha_i-\alpha_j|= |\alpha_i| - |\alpha_j|\geq \nu_{\alpha_i}(f)-\nu_{\alpha_j}(f).
		\label{eq}
	\end{equation}

	Since 
	$U:=\mathbf{A}^2_m\cap F^{\alpha_{i}}_{i,k_i}\subseteq \mathcal{H}^{\alpha_{j}}_m$ and $x^{(\alpha_{i,1})}\neq 0$ in $U$, then $\alpha_{i,1}\geq \alpha_{j,1}$ and similarly we conclude that $\alpha_{i,2}\geq \alpha_{j,2}$ as $y^{(\alpha_{i,2})}\neq 0$ in $U$. Now
	choosing $\eta=(a,b)\in supp(f)$ a vertex of the Newton polygon of $f$ such that $\eta\cdot \alpha_i=\nu_{\alpha_i}(f)$ and setting $\delta=\alpha_i-\alpha_j\in \mathbf{N}^2_0$, then $\eta\cdot \alpha_j \geq \nu_{\alpha_j}(f)$.
	Therefore,
	$$|\delta|=|\alpha_i-\alpha_j|\geq \nu_{\alpha_i}(f) - \nu_{\alpha_j}(f) \geq  \eta\cdot \alpha_i-\eta\cdot \alpha_j=\eta\cdot \delta ,$$
	and hence $a\delta_1+b\delta_2 \leq \delta_1+\delta_2$. But if $a,b\geq 1$ then $\delta\cdot (a-1,b-1)=0$ where both $\delta$ and $(a-1,b-1)\in \mathbf{R}_{\geq 0}^2$ are in the first quadrant, which is obviously a contradiction. So $a=0$ or $b=0$ but not both, since $f(0)=0$. As $a$ and $b$ play symmetric roles, then we may suppose $a=0$ with $b\neq 0$ so that $b\delta_2\leq \delta_1+\delta_2$. As $f$ defines a singularity then $b\geq 2$, since otherwise $f=y+a(x)$ can be written in Weierstrass form defining a smooth curve. Therefore, $\delta_1\geq \delta_2(b-1)\geq \delta_2$. Now, let's set $a = \nu_{\alpha_i}(f_i)+m-\nu_{\alpha_i}(f)$ and for simplicity, $\alpha = \alpha_i$ and $\beta=\alpha_j$. We have 
    $$\varnothing\neq F^{\alpha_{i}}_{i,a}\cap \mathbf{A}^2_m \subseteq \overline{F^{\alpha_j}_{j,k_j}\cap \mathbf{A}^2_m} \subseteq \mathcal{H}^{\alpha_j}_{m}\cap \{f^{[\nu_{\alpha_j}(f_j)]}_j=0\},$$
    $$(x^{(0)},..., x^{(\beta_1-1)},y^{(0)},..., y^{(\beta_2-1)},f_j^{[\nu_\beta(f_j)]})\subseteq (I: x^{(\alpha_1)}).$$
where $I = (x^{(0)},..., x^{(\alpha_1-1)},y^{(0)},..., y^{(\alpha_2-1)},f_i^{[\nu_\alpha(f_i)]},..., f^{(a)}_i)\subseteq \mathcal{O}_{\mathbf{A}^2_m}$ and 
$$(I:x^{(\alpha_1)})=\{g\in \mathcal{O}_{\mathbf{A}^2_m} : x^{(\alpha_1)}g\in I\}$$ is the colon ideal. We can write
$$x^{(\alpha_1)}f^{[\nu_\beta(f_j)]}_j = \sum_{j=0}^{\alpha_1-1}x_jX_j+\sum_{j=0}^{\alpha_2-1}y_jY_j+\sum_{j=\nu_\alpha(f_i)}^af^{(j)}_i\Phi_j.$$
Let's annihilate the variables, $$x^{(0)},..., x^{(\beta_1-1)},y^{(0)},..., y^{(\beta_2-1)},x^{(\alpha_1+1)},..., x^{(m)}, y^{(\alpha_2+1)},..., y^{(m)},$$ 
we obtain a power series ring with variables $x^{(\beta_1)},..., x^{(\alpha_1)},y^{(\beta_2)},..., y^{(\alpha_2)}$ and the expressions
$$f_i^{[\nu_\alpha(f_i)]}\equiv {[y^{(\alpha_2)}]}^{p_i}-c_i{[x^{(\alpha_1)}]}^{q_i},$$
$$f_j^{[\nu_\beta(f_j)]}\equiv {[y^{(\beta_2)}]}^{p_j}-c_j{[x^{(\beta_1)}]}^{q_j}$$
where $T(f_l) = \mathbf{R}_{\geq 0}(p_l,q_l)$ and $gcd(p_l,q_l)=1$ for $l\in \{i,j\}$. Now, if $\delta_1>0$ i.e. $\alpha_1>\beta_1$, we annihilate $x^{(\beta_1)},..., x^{(\alpha_1-1)}$ to obtain in a power series ring
$$x^{(\alpha_1)}[y^{(\beta_2)}]^{p_j}=\sum_{j=\beta_2}^{\alpha_2-1} y_jY_j+([{y^{(\alpha_2)}]}^{p_i}-c_i{[x^{(\alpha_1)}}^{q_i}]]\Phi.$$
Next, we distinguish the following 2 cases:
\begin{itemize}
    \item if $\alpha_2=\beta_2$ then inside a power series we have an expression of the form
    $$x^{(\alpha_1)}[{y^{(\beta_2)}}]^{p_j}= ([{y^{(\alpha_2)}]}^{p_i}-c_i{[x^{(\alpha_1)}}^{q_i}])\Phi,$$ 
    which is clearly not possible.
    \item if $\alpha_2>\beta_2$, then annihilate the variables $y^{(\beta_2)},..., y^{(\alpha_2-1)}$ so that in a power series ring, we have an expression of the form
    $$x^{(\alpha_1)}[{y^{(\beta_2)}}]^{p_j}= ([{y^{(\alpha_2)}]}^{p_i}-c_i{[x^{(\alpha_1)}}^{q_i}])\Psi.$$ 
\end{itemize}
    Both cases lead to a contradiction, so $0\leq \delta_2\leq \delta_1=0$. As $\delta=(0,0)$, then  $\alpha_i = \alpha_j$. Now inspecting codimensions in the inclusion:
	$$\overline{\mathbf{A}^2_m\cap F^{\alpha_i}_{i,k_i}}\subseteq \overline{\mathbf{A}^2_m\cap F^{\alpha_j}_{j,k_j}}$$
	shows that both have the same codimension (as $\alpha_i=\alpha_j$). Since each side of the inclusion is irreducible  and both have  the same codimension, then they must be equal. This completes the proof.
\end{proof} 

Our next results deal with showing that $\mathcal{H}^\gamma_m$ and $\overline{\mathbf{A}^2_m\cap F^\alpha_{i,k}}$ are not comparable by inclusion.

\begin{lemma} \label{lm : UH}
	Let $f\in \mathcal{O}_{k^2}$ define a Newton plane curve. Let $f_i$ define a branche of $f$ such that 	$$T(f_i)=\mathbf{R}_{\geq 0}. \alpha_i : \alpha_i\in \mathbf{N}^2$$ 
	and $\nu_{\alpha_i}(f)\leq m$. Set $k_i=\nu_{\alpha_i}(f_i)+m-\nu_{\alpha_i}(f)$ and let $\gamma\in \mathbf{N}^2$ such that $\nu_\gamma(f)>m$ with $\gamma$ minimal for the product order. Then, we cannot have $$\mathcal{H}^\gamma_m\supseteq \overline{\mathbf{A}^2_m\cap F^{\alpha_{i}}_{i,k_i}}.$$
\end{lemma}
\begin{proof}
If $\mathbf{A}^2_m\cap F^{\alpha_{i}}_{i,k_i}\subseteq \mathcal{H}^\gamma_m$ then $\alpha_i\geq \gamma$ for the product order. Therefore,
	$$m<\nu_\gamma(f)\leq \nu_{\alpha_i}(f)\leq m$$
	 which is a contradiction. This concludes the proof.
\end{proof}

\begin{lemma} \label{lm : HU-ok}
		Let $f\in \mathcal{O}_{k^2}$ define a Newton plane curve. Let $f_i$ define a branch of $f$ such that 	$$T(f_i)=\mathbf{R}_{\geq 0} .\alpha_i : \alpha_i\in \mathbf{N}^2$$ 
	and set $\nu_{\alpha_i}(f)= m$, $k_i=\nu_{\alpha_i}(f_i)$ and  $\gamma=\alpha_i+(1,1)\in \mathbf{N}^2$. Then, we have $$\mathcal{H}^\gamma_m\subseteq\overline{\mathbf{A}^2_m\cap F^{\alpha_{i}}_{i,k_i}}.$$
	Moreover, if $m>\nu_{\alpha_i}(f)$ then $\mathcal{H}^\gamma_m \not\subset \overline{\mathbf{A}^2_m\cap F^{\alpha_{i}}_{i,k_i+h}}$ where $h=m-\nu_{\alpha_i}(f)$.
\end{lemma}
\begin{proof}
	If $m=\nu_{\alpha_i}(f)$, then we can write
	$$\overline{\mathbf{A}^2_m\cap F^{\alpha_{i}}_{i,k_i}}=\mathcal{H}^{\alpha_i}_m\cap \{f^{(k_i)}_i=0\}$$
as both sides are irreducible with the same codimension $|\alpha_i|+1$. For the second part, suppose that $m>\nu_{\alpha_i}(f)$ and $\mathcal{H}^\gamma_m \subsetneq \overline{\mathbf{A}^2_m\cap F^{\alpha_{i}}_{i,k_i+h}}$. Then, by comparing codimensions:
$$|\alpha_i|+2 > |\alpha_i|+1+h,$$
we arrive at the inequality $1> h=m-\nu_{\alpha_i}(f)>0$ which is a contradiction.  
\end{proof}

\begin{theorem} \label{lm : HU}
	Let $f\in \mathcal{O}_{k^2}$ define a Newton plane curve. Let $f_i$ define a branch of $f$ such that 	$$T(f_i)=\mathbf{R}_{\geq 0} .\alpha : \alpha\in \mathbf{N}^2$$ 
	and $\nu_{\alpha}(f)\leq m$. Set $k=\nu_{\alpha}(f_i)+m-\nu_{\alpha}(f)$. Let $\gamma\in \mathbf{N}^2$ be minimal for the product order, such that  $\nu_\gamma(f)>m$. Moreover, assume that, $\gamma\neq \alpha+(1,1)$. Then, it holds that $$\mathcal{H}^\gamma_m\not\subset\overline{\mathbf{A}^2_m\cap F^{\alpha}_{i,k}}.$$
	
\end{theorem}
\begin{proof} 
Suppose by contradiction that  the inclusion holds. If the inclusion is an equality then $\varnothing \neq F^\alpha_{i,k}\cap \mathbf{A}^2_m\subset\mathcal{H}^\gamma_{m}$. We then have $x^{(\alpha_1)}\neq 0$ and $y^{(\alpha_2)}\neq 0$ in $\mathcal{H}^\gamma_m$ and $\gamma_1\leq\alpha_1$ and $\gamma_2\leq \alpha_2$. Since $\mathcal{H}^\gamma_m\subseteq \overline{\mathbf{A}^2_m\cap F^\alpha_{i,k}}\subseteq \mathcal{H}^\alpha_m$, then $\gamma\geq \alpha$. But $m\geq \nu_\alpha(f)=\nu_\gamma(f)>m$. This is a contradiction. So the inclusion which we assumed has to be strict. Next, we show that $|\gamma|\leq 1+|\alpha|+m-\nu_\alpha(f)$ by induction on $m\geq \nu_\alpha(f)$:
\begin{itemize}
\item At the level $m=\nu_\alpha(f)$ (since $\gamma\geq \alpha$ is minimal), $\alpha$ is deleted and replaced by $\gamma = \alpha+(1,1)$, $\gamma=\alpha+(1,0)$ or $\gamma=\alpha+(0,1)$ because if we set $n_i = \nu_\alpha(f_i)$, $r=\nu_\alpha(R)$ and $f=Rf_i$, then
$$f^{(m)}=f^{(n_i)}_i R^{(r)} \mod \alpha$$
with $Rad\langle R^{(r)}\rangle \subset\langle x^{(\alpha_1)}y^{(\alpha_2)} \cdot *\rangle$ where $*$ is a product of pairwise distinct binomials. But since $\gamma\neq \alpha+(1,1)$ then this proves that $|\gamma|=|\alpha|+1$. 
\item If $m>\nu_\alpha(f)$ and $\nu_\gamma(f)> m+1$ then $\gamma$ is still minimal at level $m+1$ and 
$$|\gamma|\leq 1+|\alpha|+m-\nu_\alpha(f)< 1+|\alpha|+m+1-\nu_\alpha(f).$$
\item If $m>\nu_\alpha(f)$ and $\nu_\gamma(f)=m+1$ then $\gamma$ (which is minimal) either bifurcates to one of these  $\{\gamma+(1,0), \ \gamma+(0,1)\}$ or  when $\gamma\in T(f)$ it gives rise to a component
$$\overline{\mathbf{A}^2_{m+1}\cap F^\gamma_{j,\nu_\gamma(f_j)}}$$
whose codimension is $|\gamma|+1$. For the first case, as $\nu_\gamma(f)<\nu_{\gamma+\epsilon}(f)$ then at level $m+2$ we can write
$$|\gamma|+1\leq 1+|\alpha|+m+2-\nu_\alpha(f)$$
where $\gamma+\epsilon$ is minimal. Let us denote $\pi:=\pi_{m+1,m}: \mathbf{A}^2_{m+1}\rightarrow \mathbf{A}^2_m$ the projection morphism. The second  case when $\gamma\in T(f)$ tells us  that
$$\mathbf{A}^2_{m+1}\cap F^{\gamma}_{j,\nu_\gamma(f_j)}=\{f_j^{[\nu_\gamma(f_j)]}=0\} \cap \pi^{-1}[\mathbf{A}^2_m\cap F^\gamma_{j,\nu_\gamma(f_j)-1}]= \{f^{[\nu_\gamma(f_j)]}_j=0\}\cap \pi^{-1}(\mathcal{H}^\gamma_m).$$
Since the inclusion $\mathcal{H}^\gamma_m\subsetneq\overline{\mathbf{A}^2_m\cap F^\alpha_{i,k}}$ is strict, then
$$\{f^{[\nu_\gamma(f_j)]}_j=0\}\cap \pi^{-1}(\mathcal{H}^\gamma_m)\subsetneq \{f^{[\nu_\gamma(f_j)]}_j=0\}\cap \overline{\mathbf{A}^2_{m+1}\cap F^\alpha_{i,k}},$$
$$\mathbf{A}^2_{m+1}\cap F^\gamma_{j,\nu_\gamma(f_j)} \subsetneq \overline{\mathbf{A}^2_{m+1}\cap F^\alpha_{i,k}}\cap \{f^{[\nu_\gamma(f_j)]}_j=0\},$$
$$\overline{\mathbf{A}^2_{m+1}\cap F^\gamma_{j,\nu_\gamma(f_j)}} \subsetneq \mathcal{H}^\alpha_{m+1}\cap\{f_i^{[\nu_\alpha(f_i)]}=...=f_i^{(k)}=f^{[\nu_\gamma(f_j)]}_j=0\}.$$
As the inclusion is strict, then  by comparing codimensions of the irreducible sets,  we obtain
$$1+|\gamma| > 2+|\alpha|+m-\nu_\alpha(f),$$ which contradicts 
  the induction hypothesis: $|\gamma|\leq 1+|\alpha|+m-\nu_\alpha(f)$. Therefore, the case  $\gamma\in T(f)$ can not occur.
\end{itemize}
Thus, we have shown $|\gamma|\leq 1+|\alpha|+m-
\nu_\alpha(f)$, and
 the proof is complete.


\end{proof}

\section{Main Theorems} \label{sec:theorem}
In this section, we describe the irreducible components of the jet schemes over the origin of a Newton non-degenerate plane curve singularity. First, we fix some notation and  present some basic definitions. 

Let $f\in \mathcal{O}_{k^2}$ be Newton non-degenerate. A \emph{tropical decomposition} of $f$ is $f=f_1...f_t$ where $T(f_i)$ is a semi-line in $\mathbf{R}_{\geq 0}^2$ passing through the origin and $T(f_i)\cap T(f_j)=\{(0,0)\}$ for $1\leq i<j\leq t$.
Next, let  $\sigma=Conv[T(f)]$, and define
$$A_{m}(i)=\{\alpha\in T(f_i)\cap \mathbf{N}^2 : \nu_\alpha(f)\leq m \}$$ 
 $$B^*_m\subseteq B_m = \min_{\leq\times \leq} \{ \gamma\in \mathbf{N}^2 : \nu_{\gamma}(f)>m\}$$
 where $B^*_m$ is obtained  from $B_m$ by removing at $m=\nu_\alpha(f)$, every $\alpha+(1,1)$ with $\alpha\in B_{m-1}\cap \partial \sigma$.
 
 The next theorem gives us the decomposition of $\mathcal{C}^0_m$ into irreducible components.

\begin{theorem} \label{th : decom}
Let $f=f_1...f_t\in \mathcal{O}_{k^2}$ be a tropical decomposition of a Newton non-degenerate polynomial. Decompose $f_i=\prod_{j=1}^{r_i} f_{(i,j)}$ into irreducible factors in $\mathcal{O}_{k^2}$.
Then, we have the  following decomposition of $\mathcal{C}^0_m$ into irreducible components 
 $$\mathcal{C}^0_m = \bigcup_{\gamma\in B^*_m} \mathcal{H}^\gamma_m \cup \bigcup_{i=1}^t \bigcup_{\alpha\in A_{m}(i)}  \bigcup_{j=1}^{r_i} \overline{ \mathbf{A}^2_m\cap F^\alpha_{(i,j),k_{ij}}}$$
 where $k_{ij}=\nu_\alpha(f_{ij})+m-\nu_\alpha(f)$.

\end{theorem}

\begin{proof}

 Recall that for $v\in \{x,y\}$,  we know  that $Cont^{>0}_{\mathcal{C}}(v)_m=\bigsqcup_{i=1}^m Cont^i_\mathcal{C}(v)_m$ where 
	$$Cont^i_{\mathcal{C}}(v)_m=\{v^{(i)}\neq 0\}\cap \{v^{(0)}=...=v^{(i-1)}=0\}\cap \{f^{(0)}=...=f^{(m)}=0\}.$$
	Now setting $Cont^{(i,j)}_\mathcal{C}(x,y)_m=Cont^i_\mathcal{C}(x)_m\cap Cont^j_\mathcal{C}(y)_m$, and observing that 
	$$\mathcal{C}^0_m=\mathcal{H}^{(1,1)}_m\cap F_m=Cont^{>0}_{\mathcal{C}}(x)_m\cap Cont_{\mathcal{C}}^{>0}(y)_m=\bigsqcup_{1\leq  i,j\leq m} Cont^{(i,j)}_\mathcal{C}(x,y)_m,$$ and
    $Cont^{(i,j)}_\mathcal{C}(x,y)_m=F^{(i,j)}_{m}\cap \mathbf{A}^2_m$  for  $\alpha=(i,j)$, then clearly we obtain
	$$\mathcal{C}^0_m=\bigsqcup_{\alpha\in \Omega} (F^\alpha_m \cap \mathbf{A}^2_m)=\left(\bigsqcup_{\alpha\in \Omega^{>}} (F^\alpha_m\cap \mathbf{A}^2_m)\right) \sqcup \left(\bigsqcup_{\alpha\in \Omega^{\leq}} (F^\alpha_m\cap \mathbf{A}^2_m) \right) $$
	where $\Omega=[1,m]^2\cap \mathbf{N}^2$, $\Omega^{>}=\{\alpha\in \Omega : \nu_\alpha(f)>m\}$ and $\Omega^{\leq }=\Omega\backslash \Omega^{>}$. Set
	$$A_m=\bigcup_{i=1}^t A_m(i).$$
	
	If $\alpha\notin A_m$  then 
	$f^{[\nu_\alpha(f)]}=c\cdot {x^{(\alpha_1)}}^{k_1}{y^{(\alpha_2)}}^{k_2} \mod \alpha$ for some constant $c$.
	As $x^{(\alpha_1)}$ and $y^{(\alpha_2)}$ are $\neq 0$ in $F^\alpha_m$ then $F^\alpha_m\cap \mathbf{A}^2_m=\varnothing$ for $\alpha\notin A_m$. Note that  we can write:
	$$\mathcal{C}^0_m=\bigsqcup_{\alpha\in \Omega^{>}} (F^\alpha_m\cap \mathbf{A}^2_m) \sqcup \bigsqcup_{i=1}^t\bigsqcup_{\alpha\in A_m(i)} (F^\alpha_m \cap \mathbf{A}^2_m).$$
	In view of Lemma \ref{lm : T},  we know that if $\alpha\in A_m(i)$ then
	$F^\alpha_m=F^\alpha_{i,k}$ where $k=\nu_\alpha(f_i)+m-\nu_\alpha(f)$, and
	$$F^\alpha_{i,k}=\bigsqcup_{j=1}^{r_i}F^{\alpha}_{(i,j),k_{ij}}.$$
		Now, as $\mathcal{H}^\alpha_m=\overline{\mathbf{A}^2_m\cap F^\alpha_m}$ for $\nu_\alpha(f)>m$, then  after ordering along $\leq \times \leq $ and using Theorem \ref{th : M1+}, we obtain:
	$$\mathcal{C}^0_m=\bigcup_{\gamma\in B^*_m} \mathcal{H}^\gamma_m\cup \bigcup_{i=1}^t \bigcup_{\alpha\in A_{m}(i)}  \bigcup_{j=1}^{r_i} \overline{ \mathbf{A}^2_m\cap F^\alpha_{(i,j),k_{ij}}}.$$ 
	The irreducibility of these components follow then from Lemma \ref{lm : inclusion}, Lemma \ref{lm : HU} and Theorem \ref{lm : UH}.
\end{proof}

It is worth mentioning here that we have just shown that $Jet(f)$ has two kinds of vertices: vertices of the form $\mathcal{H}^\gamma_m$ and others that do not vanish and continue to $\infty$. The next theorem deals with the characterization of such vertices.  

\begin{theorem} \label{th : embed}
	Let $f=f_1...f_t\in \mathcal{O}_{k^2}$ be Newton non-degenerate ; $T(f_i)\simeq \mathbf{R}_{\geq 0}$,
	$$T(f_i)\cap T(f_j)=\{(0,0)\} : i<j$$
Decompose $f_i=\prod_{j=1}^{r_i} f_{ij}$ into irreducible factors. Let $V_m\in irr(\mathcal{C}^0_m)$ with weight $(d,e)$. Then, $V_m$ is of the form $\mathcal{H}^\gamma_m$ for some $\gamma\in B^*_m$ if and only if $e=d$. Otherwise, $e>d$.
\end{theorem}
\begin{proof}
First, we set $c=codim(V_m ; \mathbf{A}^2_m)$. If $V_m=\mathcal{H}^\gamma_m$ then $c=|\gamma|$. Next, we show that $e+c=2(m+1)$. For,
$$\mathcal{O}_{V_m}= \mathcal{O}_{\mathbf{A}^2_m}/\langle x^{(0)},..., x^{(\gamma_1-1)},y^{(0)},..., y^{(\gamma_2-1)} \rangle ,$$
and $e\in \mathbf{N}$ is the least integer such that we have a surjective map of rings 
$$\mathcal{O}_{k^e}\rightarrow \mathcal{O}_{V_m}.$$
Hence, $e=dim(\mathbf{A}^2_m)-|\gamma|=2(m+1)-|\gamma|$. Thus, $e+c=2(m+1)$ i.e. $e=d$. Now, if $V_m$ is not of the form $\mathcal{H}^\gamma_m$ then there exists $\alpha\in T(f_{ij})$ for some $i,j$ such that
$$V_m=\overline{\mathbf{A}^2_m\cap F^\alpha_{ij,\nu_\alpha(f_{ij})+m-\nu_\alpha(f)}}$$
with $m\geq \nu_\alpha(f)$. Also, $c=|\alpha|+1+m-\nu_\alpha(f)$, then there are surjective ring maps:
$$\mathcal{O}_{k^e}\rightarrow \mathcal{O}_{V_m}\rightarrow \mathcal{O}_{\mathcal{H}^\alpha_m}.$$
However, this means that $e \geq 2(m+1)-|\alpha|$, and  then we can write
$$e+c\geq 2(m+1)-|\alpha|+c=2(m+1)+1+m-\nu_\alpha(f),$$
which in turn implies that $e+c>2(m+1)$. Thus, $e>d$, and the  proof is complete.
\end{proof}
Next, we  show that $Jet(f)$ determines the embedded topological type of the singularity of the Newton non-degenerate plane curve defined by $f$.

\begin{theorem} \label{th : grph}
	Let $f,g \in  \mathcal{O}_{k^2}$ be Newton non-degenerate. If the jet-components graph $Jet(f)=Jet(g)$, then $f$ and $g$ have the same embedded topological type. 
\end{theorem}
\begin{proof}
By Theorem \ref{th : embed}, we see that  $J:=Jet(f)$ determines $B^*:=\bigcup_{m\in \mathbf{N}_0}B^*_m$. In other words, $B^*_m:=B^*_m(f) = B^*_m(g)$ for all $m$.
Let $R^*$ denote the induced subgraph of $J$ which is obtained by removing the vertices of the form $\mathcal{H}^\gamma_m$ with  $\gamma\in B^*_m$ and $\nu_\gamma(f)>m\in \mathbf{N}_0$ and such that $\mathcal{H}^\gamma_m$ could not be joined with a vertex outside of $B^*$. Next, we define $A=\{\alpha : \exists m : \mathcal{H}^\alpha_m \in R^*\}$, and we note that 
$$\alpha\in A \Leftrightarrow \mathcal{H}^\alpha_{\nu_\alpha(f)-1} \in R^*, \alpha\in T(f)\cap \mathbf{N}^2.$$ Now let
 $T = \bigcup_{\alpha\in A} \mathbf{R}_{\geq 0}\cdot \alpha$ which is the tropical variety of both $f$ and $g$ and let  $f=f_1...f_r$ and $g=g_1...g_s$ be written as product of branches. In addition, for $\alpha_i\in A$, let $T_{\alpha_i} = \mathbf{R}_{\geq 0} \cdot \alpha_i$  denote the tropical variety of the branches $f_{1},..., f_{r_i},g_{1},..., g_{s_i}$ with $r_i\leq r$ and $s_i\leq s$. Since $f$ and $g$ are Newton non-degenerate with the same graph $J$ then there are exactly $r_i=s_i$ vertices above $\mathcal{H}^{\alpha_i}_{\nu_{\alpha_i}(f)-1} = \mathcal{H}^{\alpha_i}_{\nu_{\alpha_i}(g)-1}$ connected to it. Then,  $r=s$ and for $1\leq i<j\leq l<\ell\leq r$ we have
$$(f_i,f_j)_0 = \nu_{\alpha_i}(f_i) = \nu_{\alpha_i}(f_j),$$
$$(f_i,f_\ell)_0 = \nu_{\alpha_i}(f_\ell).$$
Clearly, the same holds for $g$. Now if we denote by $F = f/(f_1...f_l)$ and $G = g/(g_1...g_l)$, then
$$l\nu_{\alpha_i}(f_1)+\nu_{\alpha_i}(F)=\nu_{\alpha_i}(f)=\nu_{\alpha_i}(g)=l\nu_{\alpha_i}(g_1)+\nu_{\alpha_i}(G),$$
$\nu_{\alpha_i}(f_1) = \mbox{lcm}(\alpha_{i,1},\alpha_{i,2})= \nu_{\alpha_i}(g_1)$. Since $\nu_{\alpha_i}(f)=\nu_{\alpha_i}(g)$ then $\nu_{\alpha_i}(F)=\nu_{\alpha_i}(G)$ which holds for all $\alpha_i\in A$. Thus, we have shown that $f$ and $g$ have the same number of branches and the same intersection multiplicities (of their branches) as well as the same tropical varieties for branches (since $r_i=s_i$). Finally, in view of Theorem \ref{th : semigrp}, $f$ and $g$ have the same embedded topological type. 
\end{proof}

As we did earlier in Figures \ref{walk(2,3)}, \ \ref{fig:lambda} and \ref{walk:T(f)}, we can associate to $J=Jet(f)$ a special staircase walk $J_{SC}(f)$ in $\mathbf{R}^2_{\geq 0}$ as follows. With the same notation as in Theorem \ref{th : grph}, and with $T_i = T(f_i)$ where we assume that $T_j$ above $T_i$ for $i<j$, then  we can define a road-map  for this walk: every drawn $\alpha\in \mathbf{N}^2$ represents the irreducible components of the form $\mathcal{H}^\alpha_m$ of $\mathcal{C}^0_m$ with $m<\nu_\alpha(f)$. The red arrows correspond to the irreducible components of $\mathcal{C}^0_m$ that do not vanish and continue till $\infty$. Up to this point, we distinguish the following 3 cases.
\begin{enumerate}
	\item If $t=1$, then the staircase walk is identical to that shown in Figure \ref{fig:lambda} except this time, we add $r$ red arrows above $\alpha$ at each $\alpha\in T(f)$.
More explicitly, if we let $\lambda\in L_{(p,q)}\cap \mathbf{N}^2_0$ and with the staircase notation $[[d_1,..., d_p]]$ associated to $(p,q)$ from Theorem \ref{walk}, then our walk is given by the following moves:
\begin{itemize}
	\item The starting point is  at $(1,1)$ where $\lambda=(0,0) $,
	\item move $d_1$  unit steps vertically followed by one horizontal step,
	\item move $d_2$  unit steps vertically followed by one horizontal step,
	\item[] \ \ \ \ \ \ \  \vdots	
	\item continuing this way, we move $d_{p-1}$  unit steps vertically followed by one horizontal step,
	\item  our last move here consists of  moving $d_p-1$ vertical steps to reach $(p,q)$.
	\item Next, we add this time $r$ red arrow above $(p,q)$,
	\item  then we take $\lambda=(p,q)$ and restart again.
\end{itemize}
	Now we keep repeating the above steps for $\lambda\in L_{(p,q)}\cap \mathbf{N}_0^2$ until reaching  $\lambda+(p,q)$ above which we add $r$ red arrows.    
It is worthy to note here that each time we \textit{move} from a point $\alpha$ to  either $\alpha+(0,1)$ or $\alpha+(1,0)$, we delete $\alpha$ and  replace it by the new point. However, the red arrows are never deleted. 

	\item If $t=2$, we set $\sigma=T_1+T_2$; then the staircase walk obeys the following rules:
	\begin{itemize}
		\item Start always at $(1,1)$,
\item if $\alpha\in \overset{\circ}{\sigma}$, then replace $\alpha$ with $\alpha+(0,1)$ and $\alpha+(1,0)$,
\item if $\alpha \in T_1$, then draw $r_1$ red arrows above $\alpha$ and add $\alpha+(1,0)$,
\item if $\alpha$ is above $T_2$ then replace $\alpha$ by $\alpha+(1,0)$,
\item if $\alpha\in T_2$, then draw $r_2$ red arrows above $\alpha$ and add $\alpha+(0,1)$,
\item if $\alpha$ is below $T_1$ then replace $\alpha$ by $\alpha+(0,1)$.

\end{itemize}
	
	\item If $t>2$, we set $T=T(f)$ and $\sigma=T_1+T_t$ then the staircase walk obeys the following rules:

\begin{itemize}
		\item Start always at $(1,1)$,
\item if $\alpha\in \sigma\backslash T$, then replace $\alpha$ with $\alpha+(0,1)$ and $\alpha+(1,0)$,
\item if $\alpha\in T_i$ where $2\leq i\leq t-1$, then  draw $r_i$ red arrows above $\alpha$ and add $\alpha+(1,0)$ with $\alpha+(0,1)$,
\item if $\alpha \in T_1$, then draw  $r_1$ red arrows above $\alpha$ and add $\alpha+(1,0),$
\item if $\alpha$ is above $T_t$ then replace $\alpha$ by $\alpha+(1,0)$,
\item if $\alpha\in T_t$, then draw $r_t$ red arrows above $\alpha$ and add $\alpha+(0,1)$,
\item if $\alpha$ is below $T_1$ then replace $\alpha$ by $\alpha+(0,1)$.
\end{itemize}
\end{enumerate}

\begin{table}[htbp]
\begin{tabular}{|l|l|l|}
\hline
\textbf{Irreducible components} & \textbf{Interpretation} & \textbf{Jet-components graph} \\ \hline
                      $\mathcal{H}^{\alpha}_m : m<\nu_\alpha(f)$ &    Big, Vanishing & $[\alpha,\nu_\alpha(f)-1]$                                  \\ \hline
                     $\mathbf{A}^2_m\cap F^\alpha_{ij,k} : \alpha\in T(f_i)$  & Non-vanishing & $(\alpha,ij,\infty)$                               \\ \hline
                       
\end{tabular}
\caption{Dictionary of notation for a Newton non-degenerate  $f$.} 
\end{table} 
 
 \begin{theorem}
 	For $\mathcal{C}=\{f=0\}\subset k^2$ a Newton plane curve singularity, we can extract the irreducible components of $\mathcal{C}^0_m$ for all $m\in \mathbf{N}$ from $J_{SC}(f)$.
 \end{theorem} 
 \begin{proof}
 	Consider first an edge and denote by $(\alpha,w)$ and $(\alpha',w')$ with $w<w'<\infty $ the weighted vertices of this edge. Note that $\alpha'=\alpha+(1,0)$ or $\alpha'=\alpha+(0,1)$ with $w=\nu_\alpha(f)-1$ and $w'=\nu_{\alpha'}(f)-1$. Add on the edge $\alpha\alpha'$, $w'-w-1$ vertices representing the big components defined by the same $\alpha$. Now consider a non-vanishing component represented by $(\alpha,ij, \infty)$. Add on the arrow above $\alpha$, a countably infinite number of vertices. These represent the non-vanishing components for $m\geq \nu_\alpha(f)$ and $\alpha\in T(f)$. Finally, scale each weight $w$ by its level $m=w+1$. The obtained graph is the jet components graph.
 \end{proof}
\section{Generating Series} \label{sec:sum}
In this section, we introduce the generating series $G:=G(f)$  associated to a Newton non-degenerate power series $f\in K[[x,y]]$. If $\mathcal{C}=\{f=0\}$ denotes its associated plane curve singularity, then we define 
$$G=\sum_{K,m} u^{\text{codim}(K;\mathbf{A}^2_\infty)}\cdot v^m\in \mathbf{Z}[[u,v]]$$
where $m\geq 1$ and $K$ is an irreducible component in $\text{Cont}^m(\mathcal{C})^0\neq \varnothing$. Now given a sum $S$, then a subsummation of  $S$ is a sum whose terms form a subset of the terms of $S$. 

This section is dedicated to showing that $G=G(f)$ is rational in $u,v$, this means that $G\in \mathbf{Z}(u,v)$.
We shall always write $G=H+R$ where 
\begin{itemize}
    \item $H$ is the subsummation of $G$ with $K$ being of the form $\mathcal{H}^\beta_\infty$, and
    \item $R$ is the rest of the sum 
\end{itemize}
It is worthy to note that here we have used the letter $H$ because the $K$s appearing in $H$ are all irreducible components having hyperplane coordinates and $R=G-H$ is just the rest of the sum.

Let us start with the following two different examples.

\begin{example}
Let  $f=(y^2-x^3)(y^2-2x^3)$ whose  tropical variety $T(f)=\mathbf{R}_{\geq 0}(2,3)$ is clearly a semi-line. Set $G=G(f)$ and let 
$$R:=2\sum_{\alpha,m} u^{1+|\alpha|+m-\nu_\alpha(f)}v^m$$  
where $m\geq \nu_\alpha(f)$ and $\alpha\in T(f)$. Since $\alpha=s\cdot (2,3)$ for $s\in \mathbf{N}$, then
  $$R=2u\sum_{s=1}^{\infty} \sum_{m=12s}^\infty u^{5s+m-12s}v^{m}=\frac{2u}{1-uv}\sum_{s=1}^\infty u^{5s} v^{12s} = \frac{2u}{1-uv}\frac{u^5v^{12}}{1-u^5v^{12}}.$$
  Now, set $S_1 = (0,4)$, $S_2=(6,0)$ and let $S^*_i$ be
 the corresponding cone to $S_i$ in the dual Newton fan. In addition, let  
 $$H:=\sum_{s=1}^\infty u^{5s}v^{12s} + \sum_{\beta_1} u^{|\beta_1|}v^{\beta_1\cdot S_1}+\sum_{\beta_2} u^{|\beta_2|}v^{\beta_2\cdot S_2}$$
  where $\beta_i \in B^*_m\cap \overset{\circ}{S^*_i}, m\geq 1$. In order to describe where $\beta_i$ lies, note that 
  $$\beta_1-\lambda \in\{(1,1);(2,2)\}$$
  for $\lambda\in T(f)\cap \mathbf{N}_0^2$, and  so we get $\beta_2-\lambda\in \{(1,2)\}$ for the same $\lambda=(2s,3s)$. This means that
  $$H=\frac{u^5v^{12}}{1-u^5v^{12}}+ \sum_{s\geq 0}u^{5s+2}v^{12s+4}+\sum_{s\geq 0}u^{5s+4}v^{12s+8}+\sum_{s\geq 0}u^{5s+3}v^{12s+6}$$ 
$$H = \frac{u^5v^{12}}{1-u^5v^{12}}+\frac{u^2v^4}{1-u^5v^{12}}+\frac{u^4v^8}{1-u^5v^{12}}+\frac{u^3v^6}{1-u^5v^{12}}\in \mathbf{Z}(u,v).$$
Thus, $G=H+R\in \mathbf{Z}(u,v)$.
\end{example}

\begin{example}
Consider $f=f_1f_2=(y^2-x^3)(y^3-x^5)$ whose tropical variety is 
$$T(f)=\mathbf{R}_{\geq 0} (2,3)\cup \mathbf{R}_{\geq 0}(3,5).$$
Now we compute $$R=\underbrace{\sum_{\alpha_1,m_1}u^{1+|\alpha_1|+m_1-\nu_{\alpha_1}(f)}v^{m_1}}_{=R_1} +\underbrace{\sum_{\alpha_2,m_2}u^{1+|\alpha_2|+m_2-\nu_{\alpha_2}(f)}v^{m_1}}_{=R_2}$$
where $\alpha_i\in T(f_i)$, $m_i\geq \nu_{\alpha_i}(f)$. 
In this case,
$$R_1= u \sum_{s=1}^\infty \sum_{m=0}^\infty u^{5s+m}v^{m+15s}=\frac {u}{1-uv}\frac{u^5v^{15}}{1-u^5v^{15}}, $$ and
$$R_2 = \frac{u}{1-uv}\frac{u^8v^{24}}{1-u^8v^{24}}.$$
Next, let $S_1 = (0,5)$, $S_2 = (3,3)$ and $S_3 = (8,0)$ be the vertices of the Newton polygon from left to right. As earlier, let $S_i^*$ be the corresponding cone
to $S_i$ in the dual Newton fan. 
Then we can wrote 
$$H:=\sum_{m,\beta} u^{|\beta|}v^m=\underbrace{\sum_{r,s\geq 1} u^{5r+8s}v^{[r(2,3)+s(3,5)]\cdot (3,3)}}_{=H_{\overset{\circ}{\sigma_2}}}+\tilde{H}=\tilde{H}+\sum_{r,s} u^{5r+8s}v^{15r+24s}$$
where $\beta\in B^*_m, m\geq 1$ and $\sigma_2 = \mathbf{R}_{\geq 0}(2,3)+\mathbf{R}_{\geq 0}(3,5)=S^*_2$.
By computing the geometric series, we obtain
$$D:=H-\tilde{H} = \frac{u^5v^{15}}{1-u^5v^{15}}\frac{u^8v^{24}}{1-u^8v^{24}} \in \mathbf{Z}(u,v).$$
Next, we compute 
$$\tilde{H}= \sum_{s\geq 1}u^{5s}v^{s(2,3)\cdot (0,5)}+\sum_{s\geq 1}u^{8s}v^{s(3,5)\cdot (8,0)}+H^*$$
where $H^*$ is the subsummation of $H$ whose $\beta$ lies in the first and the last cone determined by the  orientation of the tropical decomposition.
Then, $\tilde{H}-H^*=\frac{u^5v^{15}}{1-u^5v^{15}}+\frac{u^8v^{24}}{1-u^8v^{24}}\in \mathbf{Z}(u,v)$ so that
$$H^* = \sum_{\beta_1} u^{|\beta_1|}v^{\beta_1\cdot S_1}+\sum_{\beta_3} u^{|\beta_3|}v^{\beta_3\cdot S_3}$$
where $\beta_i\in \overset{\circ}{S_i^*}\cap B^*_m, m\geq 1$, $S_1 = (0,5)$ and $S_3 = (8,0)$. Therefore, 
$$\beta_1-\lambda_1 \in \{(1,1),(2,2)\} : \lambda_1 = s_1(2,3), s_1\in \mathbf{N}_0$$
and $\beta_3-\lambda_3\in \{(1,2),(2,4)\} : \lambda_3 = s_3(3,5), s_3\in \mathbf{N}_0$. So we get $H^*$ equal to
$$ \sum_{s_1\geq 0} (u^{5s_1+2}v^{15s_1+5}+u^{5s_1+4}v^{15s_1+10})+\sum_{s_3\geq 0}(u^{8s_3+3}v^{24s_3+8}+u^{8s_3+6}v^{24s_3+16})$$
which leads to $$H^* = \frac{u^2v^5+u^4 v^{10}}{1-u^5v^{15}}+\frac{u^3v^8+u^6v^{16}}{1-u^8v^{24}}\in \mathbf{Z}(u,v).$$
Thus, $H\in \mathbf{Z}(u,v)$ and $G=H+R$ as well.
\end{example}

In the next lemma, we shall  write an expression for $R$ (which was introduced at the beginning of the section).

\begin{lemma}
Let $f\in K[x,y]$ be Newton non-degenerate whose tropical decomposition is written as $f=f_1... f_t$. Also, for $1\leq i\leq t$, let $f_i=\prod_{j=1}^{r_i}f_{i,j}$. Then,
$$R=\sum_{i,j,\alpha,m} u^{1+|\alpha|+m-\nu_\alpha(f)}v^m$$
where $1\leq i\leq t$, $1\leq j\leq r_i$, $\alpha\in T(f_i)$ and $m\geq \nu_\alpha(f)$. Moreover, $R\in \mathbf{Z}(u,v)$ and more explicitly,
$$R=\sum_{i}\frac{r_iu}{1-uv}\cdot \frac{u^{|\alpha_i|}v^{\alpha_i\cdot S_i}}{1-u^{|\alpha_i|}v^{\alpha_i\cdot S_i}}$$ 	where $T(f_i)=\mathbf{R}_{\geq 0}\cdot \alpha_i$, $gcd(\alpha_{i,1},\alpha_{i,2})=1$ and $S_{1},..., S_{t+1}$ are the vertices of the Newton polygon of $f$ (from left to right).
\end{lemma}
\begin{proof}
Set $R_i = \sum_{j,\alpha_i,m} u^{1+|\alpha_i|+m-\nu_{\alpha_i}(f)}v^m$ in such a way that $R=\sum_i R_i$. Since $\nu_{\alpha}(f) = \alpha \cdot S_i$ (because $\alpha\in T(f_i)$), then we obtain:
$$R_i = r_i\cdot \sum_{s\geq1} \sum_{m\geq 0} u^{1+s|\alpha_i|+m}v^{m+s\cdot \nu_{\alpha_i}(f)}=\frac{r_iu}{1-uv}\frac{u^{|\alpha_i|}v^{\alpha_i\cdot S_i}}{1-u^{|\alpha_i|}v^{\alpha_i\cdot S_i}}$$
which completes the proof.	
\end{proof}

Now let $f=f_1...f_t$ be a tropical decomposition of a Newton non-degenerate polynomial. Then, the  decomposition $f=f_1...f_t$ is \emph{oriented} in such a way that for all $i<j$, 
$$det[T(f_i), T(f_j)]>0$$
i.e. $T(f_j)$ is "above" $T(f_i)$ for which we write $T(f_j)>T(f_i)$.
Making use of this new concept of oriented tropical decomposition, we obtain the following observation.

\begin{observation}
\label{obs : uv}
    Let $f$ define a Newton Plane Curve singularity such that $(1,1)\notin supp(f)$. Then, for every $\alpha\in T(f)$, we have $\nu_\alpha(f)>|\alpha|$ unless $\alpha\in \mathbf{R}_{\geq 0}(1,1)$ with $f$ having two branches and $T(f) = \mathbf{R}_{\geq 0}\cdot (1,1)$.
\end{observation}
\begin{proof}
    Due to the symmetry $(x,y)\mapsto(y,x)$, then without loss of generality we assume that  $\alpha = (m,n)$ with $m\geq  n\geq 1$ and $\gcd(m,n)=1$. Suppose that
    $$\nu_\alpha(f_\alpha)\leq \nu_\alpha(f)\leq |\alpha|$$
    where $f_\alpha$ denotes the product of branches having tropical variety $\mathbf{R}_{\geq 0}\cdot \alpha$. 
    Now we  show that $f=f_\alpha$. If not, then $t\geq 2$ where $t$ is the number of semi-lines in the tropical variety $T(f)$. Denote by $f=f_1...f_t$ an oriented tropical decomposition of $f$. Since  $f_\alpha\in \{f_1,..., f_t\}$, then there exists $S=(a,b)$ a vertex in the Newton polygon of $f$ with $a,b\geq 1$ (as $t>1$) and $\nu_{\alpha}(f)=\alpha\cdot S$. Next, set $\eta=(1,1)-S$ then
    $$0\leq |\alpha|-\nu_\alpha(f) = \alpha\cdot \eta.$$
    Note that as $\alpha_1,\alpha_2\geq 1$, 
    and $a,b\geq 1$, then  this last inequality implies that $\alpha\cdot \eta=0$. However, $\alpha$ being in the first positive quadrant $\mathbf{R}^2_{> 0}$ and $\eta$ being in the third quadrant $\mathbf{R}^2_{\leq 0}$, they can not be orthogonal unless $\eta=0$ and then $S=(1,1)\in supp(f)$ which is a   contradiction. Therefore,  $t=1$ and $f=f_\alpha$. Observe that
    $$mn\leq rmn\leq m+n$$
    where $r$ is the number of branches of $f$. In view of Observation \ref{obs:branch}, $n=1$  which in turn implies that
    $$1\leq r \leq 1+\frac{1}{m}.$$
    If $m>1$, then $r=1$ and we can write in Weierstrass form $$f = y+a(x)$$
    and this leads to $f$ being smooth, which contradicts the fact that  $f$ defines a singularity. Therefore, $r=2$ and $m=1$. So that after performing a change of variables, we can write in Weiersrass form 
    $$f=(y-x+...)(y-cx+...).$$
    Thus, $|\alpha|=\nu_\alpha(f)=2$ with $\alpha=(1,1)$, and this concludes the proof of our observation.
\end{proof}

In the next lemma, we shall give an expression for $H$ (which was introduced
at the beginning of the section). It is worth recalling here the following:
\begin{itemize}
    \item if $T=\cup_i T_i$ is the tropical variety of $f$, then $H_{T_i}$ is the subsummation of $H$ where $\beta\in T_i$;
    \item if $\sigma_i=T_{i-1}+T_i$ then $H_{\overset{\circ}{\sigma_i}}$ is the subsummation of $H$ whose $\beta$ lie in the interior $\overset{\circ}{\sigma_i}$ of the cone.
\end{itemize}

\begin{lemma} \label{lm : sum}
	Let $f\in K[x,y]$ be a Newton non-degenerate polynomial whose oriented tropical decomposition is denoted $f=f_1...f_t$. Then $H$ is given by
	$$H=\sum_{m,\beta}u^{|\beta|}v^m$$
	where $m\geq 1$, $Cont^m(\mathcal{C})^0\neq \varnothing$ and $\beta\in B^*_m$. Also, $H\in \mathbf{Z}(u,v)$. Moreover, if $S_1,..., S_{t+1}$ are the vertices of the Newton polygon of $f$ from left to right, then, the following holds.
	\begin{itemize}
		\item If $t=1$ and $T(f)=\mathbf{R}_{\geq 0}(1,1)$, then $H = \frac{u^2v^s}{1-u^2v^s}$ where the Newton polygon of $f$ is the line=segment joining $(s,0)$ and $(0,s)$,
	\item If $t=1$ and $T(f)>\mathbf{R}_{\geq 0}(1,1)$, then $H = H_T + H_{\overset{\circ}{\sigma_1}}+H_{\overset{\circ}{\sigma_2}}$ with
	$$H_T = \frac{u^{|\alpha|}v^{\alpha\cdot S}}{1-u^{|\alpha|}v^{\alpha\cdot S}},$$
	and where $T(f)=\mathbf{R}_{\geq 0}\cdot \alpha$, $\gcd(\alpha_1,\alpha_2)=1$ and $S=(s,0)$ is a vertex of the Newton polygon of $f$ as well as
	$$H_{\overset{\circ}{\sigma_1}}=\frac{1}{1-u^{|\alpha|}v^{\alpha\cdot S}}\sum_{i=1}^p\sum_{j=D_{i-1}+1}^{D_i} u^{i+j}v^{js},$$
	$$H_{\overset{\circ}{\sigma_2}}=\frac{1}{1-u^{|\alpha|}v^{\alpha\cdot S}}\sum_{i=1}^p u^{i+D_i+1}v^{is}$$
	where $D_0 = 0$, $D_i = \sum_{j=1}^i d_j$ and the $d_i$'s  come from $SC(\alpha_2/\alpha_1)=[[d_1,..., d_p]]$.
	\item If $t>1$, then by setting  $T_i = T(f_i)=\mathbf{R}_{\geq 0}\cdot \alpha_i$ with $T_1<...< T_t$ and $T_1=\mathbf{R}_{\geq 0}(1,1)$ we obtain 
	$$H = H_{T_1}+...+H_{T_t}+H_{\overset{\circ}{\sigma_2}}+...+H_{\overset{\circ}{\sigma_{t+1}}}$$
	where $H_{T_i} = \frac{u^{|\alpha_i|}v^{\alpha_i\cdot S_i}}{1- u^{|\alpha_i|}v^{\alpha_i\cdot S_i}}$ and 
	$$H_{\overset{\circ}{\sigma_i}}=\frac{u^{|\alpha_i|}v^{\alpha_i\cdot S_i}}{1-u^{|\alpha_i|}v^{\alpha_i\cdot S_i}}\cdot \frac{u^{|\alpha_{i-1}|}v^{\alpha_{i-1}\cdot S_{i-1}}}{1-u^{|\alpha_{i-1}|}v^{\alpha_{i-1}\cdot S_{i-1}}}$$
	  for $2\leq i\leq t$ as well as 
	$$H_{\overset{\circ}{\sigma_{t+1}}}=\frac{1}{1-u^{|\alpha_{t}|}v^{\alpha_t\cdot S_{t+1}}}\sum_{i=1}^{p_t} u^{i+D_{i,t}+1}v^{i\cdot s_{t+1}}$$
	where $S_{t+1}=(s_{t+1},0)$,  $D_{0,t}=0$, $D_{i,t}=\sum_{j=1}^i d_{j,t}$  and the $d_i$'s again come from $SC(\alpha_{2,t}/\alpha_{1,t})=[[d_{1t},..., d_{p_t,t}]]$. 
	\item If $t>1$, then by setting $T_i =\mathbf{R}_{\geq 0}\cdot \alpha_i$ with $\mathbf{R}_{\geq 0}.(1,1)<T_1<...< T_t$ we get
	$$H=H_{T_1}+...+H_{T_t} + H_{\overset{\circ}{\sigma_1}}+...+H_{\overset{\circ}{\sigma_{t+1}}}$$
	with $H_{T_i} = \frac{u^{|\alpha_i|}v^{\alpha_i\cdot S_i}}{1-u^{|\alpha_i|}v^{\alpha_i\cdot S_i}}$, and  $H_{\overset{\circ}{\sigma_i}}$ is given by the previous case for all  $2\leq i\leq t+1$ as well as 
	$$H_{\overset{\circ}{\sigma_1}}=\frac{1}{1-u^{|\alpha_1|}v^{\alpha_1\cdot S_1}}\sum_{i=1}^{p_1}\sum_{j=D_{i-1,1}+1}^{D_{i,1}} u^{i+j}v^{js_1}$$
	where $S_1 = (0,s_1)$.
	\item If $t>1$ with $T_1<\mathbf{R}_{\geq 0} (1,1)<T_t$ then 
	$$H = H_{T_1}+...+H_{T_t}+H_{\overset{\circ}{\sigma_1}}+...+ H_{\overset{\circ}{\sigma_{t+1}}}$$ where
	$H_{T_i}$, $H_{\overset{\circ}{\sigma_i}}$ are as before for all $2\leq i\leq t+1$ and with
	$$H_{\overset{\circ}{\sigma_1}} = \frac{1}{1-u^{|\alpha_1|}v^{\alpha_1\cdot S_1}}\sum_{i=1}^{p_1} u^{i+D_{1,i}+1}v^{i\cdot s_1}$$
	where $D_{1,i} = \sum_{j=1}^i d_{1,j}$, $SC(\alpha_{1,1}/\alpha_{2,1})=[[d_{11},..., d_{p_1,1}]]$.
	\end{itemize}
\end{lemma}
\begin{proof}
The first part comes from the fact that $H$ has the decomposition:
$$H = \sum_{i=1}^t H_{T_i} +\sum_{i=1}^{t+1} H_{\overset{\circ}{\sigma_i}}$$
where $H_{T_i}$ is the sum over the lattice points which lie on $T_i\cap \mathbf{N}^2$ and $H_{\overset{\circ}{\sigma_i}}$ is the sum over the lattice points which lie in the interior of the cone $\sigma_i = T_{i-1} + T_{i}$. More explicitly,
$$H_{T_i} = \sum_{s=1}^\infty u^{s|\alpha_i|}v^{s\alpha_i\cdot S_i} = \frac{u^{|\alpha_i|}v^{\alpha_i\cdot S_i}}{1-u^{|\alpha_i|}v^{\alpha_i\cdot S_i}}$$	and for $2\leq i\leq t$, 
$$H_{\overset{\circ}{\sigma_i}} = \sum_{r,s\geq 1} u^{r|\alpha_i|+s|\alpha_{i-1}|}v^{(r\alpha_i+s\alpha_{i-1})\cdot S_i} = \frac{u^{|\alpha_i|}v^{\alpha_i\cdot S_i}}{1-u^{|\alpha_i|}v^{\alpha_i\cdot S_i}} \cdot \frac{u^{|\alpha_{i-1}|}v^{\alpha_{i-1}\cdot S_{i-1}}}{1-u^{|\alpha_{i-1}|}v^{\alpha_{i-1}\cdot S_{i-1}}}$$
as $\alpha_{i-1}\cdot S_i = \alpha_{i-1}\cdot S_{i-1}$ for $\alpha_i\in \sigma_i$. For $H_{\overset{\circ}{\sigma_1}}$ and $H_{\overset{\circ}{\sigma_{t+1}}}$, there are multiple cases depending on whether $T(f_1)\geq \mathbf{R}_{\geq 0}(1,1)$ or $T(f_1)<\mathbf{R}_{\geq 0}(1,1)<T(f_t)$. 
\begin{itemize}
    \item If $T(f_1)=\mathbf{R}_{\geq 0}(1,1)$, then $H_{\overset{\circ}{\sigma}_1} = 0$.
    \item If on the other hand $T(f_1)>\mathbf{R}_{\geq 0}(1,1)$ and $t=1$, then choose $\alpha=(p,q)$ as the generator of the tropical variety of $f$ such that $\gcd(p,q)=1$ and  $SC(q/p) = [[d_1,..., d_p]]$. In addition,
    $$\overset{\circ}{\sigma}_1\cap B^* = \{\lambda+u : \lambda\in T(f)\cap \mathbf{N}^2_0, u\in U\}$$
    where $U = \{(1,1),..., (1,d_1), (2,d_1+1),..., (2,d_1+d_2),..., (p,d_1+...+d_{p-1}+1),..., (p,q-1)\}$ and we have
    $$H_{\overset{\circ}{\sigma}_1} = \sum_{\lambda\in T(f)\cap \mathbf{N}_0^2} \:\ \sum_{\nu\in U} u^{|\lambda|+|\nu|}v^{[\lambda + \nu]\cdot S_1}$$
    where $S_1=(0,s_1)$ is the first vertex of the Newton polygon of $f$. As for $H_{\overset{\circ}{\sigma}_2}$, the following holds
    $$H_{\overset{\circ}{\sigma}_2} = \sum_{\lambda\in T(f)\cap \mathbf{N}_0^2} \:\ \sum_{\nu\in V} u^{|\lambda|+|\nu|}v^{[\lambda + \nu]\cdot S_2}$$
    with $V = \{(1,d_1+1), (2,d_1+d_2+1),...,(p-1, d_1+...+d_{p-1}+1)\}$.
    \item If $T(f_1)>\mathbf{R}_{\geq 0}(1,1)$ and $t>1$, then $H_{\overset{\circ}{\sigma}_1}$ (resp. $H_{\overset{\circ}{\sigma}_{t+1}}$) can be computed in the same way as $H_{\overset{\circ}{\sigma}_1}$ (resp. $H_{\overset{\circ}{\sigma}_2}$) in the previous case but with replacing $T(f)$ by $T(f_1)$ (resp. by $T(f_t)$) and $SC(q/p)$ by $SC(\alpha_{1,2}/\alpha_{1,1})$ (resp. $SC(\alpha_{t,2}/\alpha_{t,1})$).
    \item If $T(f_1) < \mathbf{R}_{\geq 0}(1,1)< T(f_t)$, then $H_{\overset{\circ}{\sigma}_{t+1}}$ can be found in the same way  as $H_{\overset{\circ}{\sigma}_{t+1}}$ in the previous case too. As for $H_{\overset{\circ}{\sigma}_1}$ it is given by:
    $$H_{\overset{\circ}{\sigma}_1} = \sum_{\lambda\in T(f)\cap \mathbf{N}_0^2} \:\ \sum_{u'\in U'} u^{|\lambda|+|u'|}v^{[\lambda + u']\cdot S_1}$$ where
    $U'$ is the same set as $U$ but where axes and ordinates of each lattice point are flipped. 
\end{itemize}
Now a simple check shows that each of $H_{\overset{\circ}{\sigma}_1}$ and $H_{\overset{\circ}{\sigma}_{t+1}}$ are given  as in the statement of the theorem. 
Finally, observe  that the case where $T(f_1)>T(f_t)$ can not occur since we have assumed that
$$f=f_1...f_t$$ 
is an oriented tropical decomposition, and  the case $T(f_1)\leq T(f_t)\leq\mathbf{R}_{\geq 0}(1,1)$ can be dealt with by  flipping to $\mathbf{R}_{\geq 0}(1,1)\leq T(g_1)\leq T(g_t)$
where $g_i(x,y) = f_i(y,x)$ for all $i$.


\end{proof} 
Our next goal is to show that  for a  Newton non-degenerate $f$, $G=G(f)$ is rational and we shall compute its  poles. Recall that $-\frac{b}{a}$ with $\gcd(a,b)=1$ is a \emph{pole} of the generating series $G$ of $f$ if we can write $G = \frac{N}{D}$ with $N,D \in \mathbf{Z}[u,v]$, $\gcd(N,D)=1$ and $1-u^av^b$ divides $D$. 

First, we start with the following lemmas which will be used in the proof of Theorem \ref{th : fin}.

\begin{lemma} \label{lm : pole}
    Let $G=H+R$ be a sum of $H,R\in \mathbf{Z}[u,v]$ such that $p$ is a pole of $R$ but not a pole of $H$. Then $p$ is a pole of $G$.
\end{lemma}
\begin{proof}
    Let $p=-\frac{b}{a}$ with $\gcd(a,b)=1$, and denote by $H = \frac{N_H}{D_H}$ and $R = \frac{N_R}{D_R}$ with $\gcd(N_H,D_H)=\gcd(N_R,D_R)=1$. Since $p$ is a pole of $R$  but not a pole of $H$, then $1-u^av^b$ divides $D_R$ and $1-u^av^b$ does not divide $D_H$. Now by writing $D_R = (1-u^av^b)\Delta_R$, then
    $$G = H+R = \frac{N_HD_R+N_RD_H}{(1-u^av^b)\Delta_R D_H}.$$
    If $p$ is not a pole of $G$, then $1-u^av^b$ divides $N_HD_R+N_RD_H$, which in turn means that $1-u^av^b$ divides $N_RD_H$. However, $1-u^av^b$ is irreducible in $\mathbf{Z}[u,v]$ since $\gcd(a,b)=1$. Since $p$ is not a pole of $H$, then $\gcd(1-u^av^b, D_H)=1$. Therefore, $1-u^av^b$ divides $N_R$. So that $1-u^av^b$ divides both $N_R$, $D_R$ which are relatively prime, which is a contradiction. Thus, the proof is complete.
\end{proof}

\begin{lemma} \label{lm : unit}
    In $R = \mathbf{Z}[w,w^{-1}]$, the units are $R^\times = \pm w^{\mathbf{Z}}$.
\end{lemma}
\begin{proof}
    First note that $\pm w^\mathbf{Z} \subseteq R^\times $. Let $f\in R^\times$ and denote by $f'$ its inverse in $R$. Also, let
    $f = \sum_{k=m}^M f_kw^k$ and $f' = \sum_{i=m'}^{M'} f'_iw^i$. Since $f'f=1$, then $m+m' = M +M'=0$ which can be seen from comparing the least and largest exponents of $f'f=1$. Moreover, $f$ must have one exponent $f=f_nw^n$ because otherwise there would be at least two distinct exponents in the product. As $f_n\in \mathbf{Z}^\times =\{-1,1\}$, then $f=\pm w^n$ with $n\in \mathbf{Z}$. This concludes our proof.
\end{proof}

Now, we end with our final theorem concerning the poles of the generating series.

\begin{theorem} \label{th : fin}
Let $f$ define a Newton non-degenerate plane curve singularity. Denote by $G$ its generating series. Then, $G=H+R$ (from the previous lemmas) is in $\mathbf{Z}(u,v)$ and its poles are given by:
$$\{-1\}\cup \{\frac{-\alpha_l\cdot S_l}{|\alpha_l|}:1\leq l\leq t\}.$$
\end{theorem}
\begin{proof}
Recall first that $\alpha_i\cdot S_i = \nu_{\alpha_i}(f)$. Now $R = \frac{u}{1-uv} \sum_{i=1}^t r_iH_{T_i}$ where $H_{T_i} = \frac{u^{|\alpha_i|}v^{\nu_{\alpha_i}(f)}}{1-u^{|\alpha_i|}v^{\nu_{\alpha_i}(f)}}$, and  $H = \sum_{i=1}^t H_{T_i} + \sum_{i=1}^{t+1}H_{\overset{\circ}{\sigma_i}}$ where for $2\leq i\leq t$, $H_{\overset{\circ}{\sigma_i}} = H_{T_i}H_{T_{i-1}}$ and for $j\in \{1,t+1\}$, set $\alpha_t = \alpha_{t+1}$ and write
    $$H_{\overset{\circ}{\sigma_j}} = \frac{P_j(u,v)}{1-u^{|\alpha_j|}v^{\nu_{\alpha_j}(f)}}$$
    where $P_j\in \mathbf{C}[u,v]$. Then, the set of poles $\mathcal{P}$ of $G$ verifies 
    $$\mathcal{P}\subseteq \{-1\}\cup \{\frac{-\alpha_l\cdot S_l}{|\alpha_l|}:1\leq l\leq t\}.$$
    \begin{itemize}
        \item If $T(f) = \mathbf{R}_{\geq 0}(1,1)$ with $f$ having two branches, then choose coordinates $(x,y)$ such that
        $$f=(y-x+...)(y-cx+...).$$
        Then $s=2$ in
        $$G = \frac{u^2v^s}{1-u^2v^s}+\frac{2u}{1-uv}\cdot\frac{u^2v^2}{1-u^2v^2}=\frac{u^2v^2(1-uv+2u)}{(1-u^2v^2)(1-uv)}$$
        so that $G$ has $-1$ as the only pole. 
        \item Otherwise,  if also $(1,1)\notin supp(f)$, then by Observation \ref{obs : uv},  $\nu_\alpha(f)>|\alpha|$ for all $\alpha\in T(f)$. Hence, 
        $$G=H+R$$ with $R$ having $-1$ as a pole. In fact, 
        $$R = \frac{u}{1-uv}\sum_{i=1}^t \frac{ r_i\cdot u^{|\alpha_i|}v^{\alpha_i\cdot S_i}}{1-u^{\alpha_i}v^{\alpha_i\cdot S_i}}$$
        which after performing computation leads to
        $$R=\frac{u\sum_{i=1}^t r_i\cdot u^{|\alpha_i|}v^{\alpha_i\cdot S_i}\prod_{j\neq i} (1-u^{|\alpha_j|}v^{\alpha_j\cdot S_j})}{(1-uv)\prod_{l=1}^t (1-u^{|\alpha_l|}v^{\alpha_l\cdot S_l})}.$$
        Thus, $-1$ is a pole of $R$ since $1-uv$ does not divide its numerator.
As for $I = H-H_{\overset{\circ}{\sigma}_1}-H_{\overset{\circ}{\sigma}_{t+1}}$, it does not have $-1$ as a pole since after a simple check we can write
$$I = \frac{q_1\prod_{j=2}^t (1-q_j)+\sum_{i=2}^t q_i\prod_{j\neq i,i-1}(1-q_j)}{\prod_{l=1}^t (1-q_l)}$$
where $q_j = u^{|\alpha_j|}v^{\alpha_j\cdot S_j}\neq u^kv^k$ for all $k$ in the case of $I$. As for $H_{\overset{\circ}{\sigma_1}}$ and $H_{\overset{\circ}{\sigma_{t+1}}}$ they are of the form
$$\frac{P(u,v)}{1-u^{|\alpha|}v^{\nu_\alpha(f)}}$$
of course with $|\alpha|<\nu_\alpha(f)$ so that  $-1$ is not a pole for both. This in turn means that $H$ does not have $-1$ as a pole and since $R$ has $-1$ as a pole, then $G$ has $-1$ as a  pole in view of Lemma \ref{lm : pole}. As for $\frac{-\alpha_l\cdot S_l}{|\alpha_l|}\neq -1$ where $1\leq l\leq t$, then we can write

$$G = I+\frac{P_1}{1-q_1}+\frac{P_{t+1}}{1-q_{t+1}}+\sum_{i=1}^t\frac{r_iq_iu}{1-q_i},$$

$$I = \frac{N_I}{\prod_{j=1}^t (1-q_j)} = \frac{q_1}{1-q_1}+\sum_{i=2}^t \frac{q_i}{(1-q_i)(1-q_{i-1})}.$$
Set $L = \{i\in [1,t+1] : \frac{\alpha_l\cdot S_l}{|\alpha_l|} = \frac{\alpha_i\cdot S_i}{|\alpha_i|}\}$.
Now there exists $\frac{-b}{a} = \frac{-\alpha_l\cdot S_l}{|\alpha_l|}$ with $\gcd(a,b)=1$ so that we can write
$$l\in L = \{i\in [1,t+1] : 1-u^av^b \text{ divides }1-q_i\} = \{i_1<...<i_\lambda\}.$$ where $\lambda = \#L$. Next, we shall set $$N = (P_1+q_1)\prod_{k=2}^{t+1}(1-q_k)+\sum_{i=2}^t q_i\prod_{k\neq i,i-1} (1-q_k)+P_{t+1}\prod_{k=1}^t (1-q_k)+\sum_{i=1}^t r_iq_iu\prod_{k\neq i} (1-q_k),$$
$D = \prod_{k=1}^{t+1}(1-q_k)$ and keep in mind that  $G = \frac{N}{D}$. Next, we consider the following cases:
\begin{itemize}
    \item[$\blacksquare$] if $i_1\neq 1$, then we write 
    $$D = (1-u^av^b)^\lambda \Delta$$
    where $1-u^av^b \nmid \Delta$. We have
    $$N \equiv r_lu\prod_{k\neq l} (1-q_k)+\prod_{k\neq l-1,l} (1-q_k)+q_{l+1}\prod_{k\neq l,l+1} (1-q_k)\mod (1-u^av^b).$$
    There exist $\varepsilon_1,\varepsilon_2\in\{0,1\}$, $\Delta_i\neq 0 \mod (1-u^av^b)$ for $i\in \{l-1,l,l+1\}$ such that, 
    $$N\equiv r_lu(1-u^av^b)^{\lambda-1}\Delta_l+(1-u^av^b)^{\lambda-1-\varepsilon_1}\Delta_{l-1}+(1-u^av^b)^{\lambda-1-\varepsilon_1}\Delta_{l+1}$$
    modulo $1-u^av^b$. Since $\lambda-1-\varepsilon_i<\lambda$, we get that $G$ has $-\frac{b}{a}$ as a pole.

    \item[$\blacksquare$] If $i_1=1$, $2\leq l\leq t$, then
    $$N\equiv q_2\prod_{k\neq 1,2}(1-q_k)\mod (1-u^av^b)$$
    but as in the previous case,  this becomes
    $$N\equiv (1-u^av^b)^{\lambda-1-\varepsilon}\Delta_{1,2} \mod (1-u^av^b)$$
    where $\Delta_{1,2}$ is not divisible by $1-u^av^b$. Since $\lambda-1-\varepsilon<\lambda$, then $-b/a$ is again a pole of $G$.
   \item[$\blacksquare$] If $\lambda>1$ and  $i_1=l=1$, then
   $$N\equiv q_2\prod_{k\neq 1,2}(1-q_k) \equiv q_2(1-u^av^b)^{\lambda-2}\Delta_2\mod (1-u^av^b)$$
   which implies as before that $G$ has $-b/a$ as a pole.
   \item[$\blacksquare$] If $\lambda=i_1=l=1$, then
   $$N\equiv (P_1+1+r_1u)\prod_{k=2}^{t+1}(1-q_k)+q_2\prod_{k\neq 1,2}(1-q_k)\mod (1-u^av^b)$$
   $$ \ \equiv \prod_{k\neq 1,2}(1-q_k)\cdot [q_2+(1-q_2)(P_1+1+r_1u)] \mod (1-u^av^b)$$
   $$ \ \  \ \equiv (1-u^av^b)^{\lambda-1}\Delta'\cdot [q_2+(1-q_2)(P_1+1+r_1u)] \mod (1-u^av^b)$$
   with $1-u^av^b\nmid  \Delta'$. Suppose by contradiction that $$q_2 \equiv (q_2-1)(P_1+1+r_1u) \mod \mathfrak{p}$$ (where $\mathfrak{p} =  (1-u^av^b)$). In the quotient ring $\mathbf{Z}[u,v]/\mathfrak{p}$, we have 
   $$q_2 = (q_2-1)(P_1+1+r_1u)$$
   so that $q_2-1\mid q_2$. However, in this case
   $$(q_2-1) = (q_2, q_2-1) = \mathbf{Z}[u,v]/\mathfrak{p}=\mathbf{Z}[v,v^{-b/a}].$$
   Hence, $1-q_2 = u^{|\alpha_2|}v^{\alpha_2\cdot S_2}$ is a unit in the quotient ring.  
   Moreover, 
   $$\mathbf{Z}[v,v^{-b/a}] = \sum_{r,s\in \mathbf{N}_0} \mathbf{Z}\cdot v^{\frac{1}{a}(ra-sb)}.$$
   On the other hand, by B\'ezout's theorem, there exist $r,s\in \mathbf{N}_0$ such that $ra-sb = \pm 1$. 
   \begin{itemize}
       \item if $ra-sb=1$ then $\mathbf{Z}[v,v^{-b/a}]=\mathbf{Z}[v^{1/a}]$ and $v^{b}$ is not invertible which is a  contradiction.
       \item if $ra-sb=-1$ then $\mathbf{Z}[v,v^{-b/a}]=\mathbf{Z}[v,v^{-1/a}]$. Next if we  set $w=v^{1/a}$, then we get $u^aw^{ab}=1$ and
       $$\mathbf{Z}[v,v^{-1/a}]=\mathbf{Z}[w^a,w^{-1}]=\mathbf{Z}[w,w^{-1}].$$
       \end{itemize}
   Now $1-q_2$ is a unit in the ring $\mathbf{Z}[w,w^{-1}]$ so that $q_2 = 1+\varepsilon w^n$ with $\varepsilon = \pm1$, $n=-b|\alpha_2|+a(\alpha_2\cdot S_2)\in \mathbf{Z}$ by making use of Lemma \ref{lm : unit}. So in $\mathbf{Z}[u,v]/\mathfrak{p} = \mathbf{Z}[w,w^{-1}]$:
    $$q_2+(1-q_2)(P_1+1+r_1u)= 0, \text{ hence }$$
    $$P_1+1+r_1u = \varepsilon w^{-n}(1+\varepsilon w^n)=1+\varepsilon w^{-n},$$ which gives that
    $$P_1 = \varepsilon w^{-n}-r_1u.$$
    As $u^aw^{ab}=1$ then $u = \varepsilon_u w^{-b}$  with $\varepsilon_u=\pm 1$. So we arrive at,
    $$P_1  =\sum_{\nu\in U_1} u^{|\nu|}v^{\nu\cdot S_1}=\sum_\nu (\varepsilon_u)^{|\nu|} w^{a(\nu\cdot S_1)-b|\nu|}$$
    where $U_1\subset \mathbf{N}^2$ is a set of lattice points. 
Now  $$\varepsilon w^{-n}-r_1\varepsilon_uw^{-b}=P_1=\sum_\nu (\varepsilon_u)^{|\nu|} w^{a(\nu\cdot S_1)-b|\nu|}\in \mathbf{Z}[w,w^{-1}]$$ and
$1-u^av^b = 1-\varepsilon_u w^{-ab+ab}=1-\varepsilon_u$. Then $\varepsilon_u = 1$ because $1-u^av^b = 0\in \mathbf{Z}[w,w^{-1}]$. Next, we split the 2 cases: $P_1\neq 0$ and $P_1=0$.
 \begin{enumerate}
    \item if $P_1\neq 0$ then $\alpha_1\neq (1,1)$ and there exists $\beta\in U_1$ with
    $$a(\beta\cdot S_1)-b|\beta|=-b,$$
    $$\frac{\alpha_1\cdot S_1}{|\alpha_1|}=b/a=\frac{\beta\cdot S_1}{|\beta|+1}.$$
 As $-b=a(\beta\cdot S_1)-b|\beta|$, then
 $$a(\beta\cdot S_1)<b|\beta|<b(1+|\beta|),$$
 $$b/a=\frac{\beta\cdot S_1}{1+|\beta|}<\frac{b}{a}\cdot \frac{|\beta|}{1+|\beta|}<b/a$$ which is clearly
 a contradiction.
 
\item if $P_1=0$ then $\varepsilon = r_1 = 1$, $b=n$ and $\alpha_1=(1,1)$ so that
$$uq_2 = u(1+w^n) = u(1+w^b)=u(1+u^{-1}) = 1+u$$
in the quotient ring. Then, 
$$uq_2 = 1+u-(1-u^av^b)Q,$$
for some $\mathbf{Z}[u,v]\ni Q\equiv 1 \mod u$. If we  write $Q=1+uU$, then
$$u^{|\alpha_2|}v^{\alpha_2\cdot S_2}=1-q_2=(1-u^av^b)U-u^{a-1}v^b.$$
Let us set $i=|\alpha_2|$ and $j=\alpha_2\cdot S_2$ in 
    $$u^iv^j+u^{a-1}v^b = w^{aj-bi}+w^{b}\neq 0.$$
    Since $1-u^av^b=0\in \mathbf{Z}[w,w^{-1}]$ and $U\in \mathbf{Z}[u,v]$ then we must have $u^iv^j+u^{a-1}v^b =0\in \mathbf{Z}[w,w^{-1}]$ which  is not the case. Hence we arrived at a  contradiction. 
 \end{enumerate}
   Therefore, in all cases, $1-u^av^b$ does not divide $q_2+(1-q_2)(P_1+1+r_1u)$. Thus, $N\equiv (1-u^av^b)^{\lambda-1}\Delta'' \mod (1-u^av^b)$ with $1-u^av^b \nmid \Delta''$. So that $G$ has $-b/a$ as a pole.

\end{itemize}

        \item Finally, if $(1,1)\in supp(f)$ then the Newton Polygon of $f$ has two faces (i.e. $t=2$) with $(1,1)$ as a middle vertex. Moreover, $f$ has two branches since $(1,1)\in supp(f)$ and for all $\alpha\in T(f)$, $\nu_\alpha(f)=|\alpha|$. So $\alpha$ and $S=(1,1)$ should correspond to the pole $-1=\frac{-\alpha\cdot S}{|\alpha|}$. Since $T(f) = \mathbf{R}_{\geq 0}\cdot \alpha\cup \mathbf{R}_{\geq 0}\cdot \beta$ where $\alpha=(a,1)$, $\beta=(1,b)$  (as $(1,1)\in supp(f)$), then
        $$R = (1+\frac{u}{1-uv})(\frac{u^{|\alpha|} v^{|\alpha|}}{1-u^{|\alpha|}v^{|\alpha|}}+\frac{u^{|\beta|} v^{|\beta|}}{1-u^{|\beta|}v^{|\beta|}})$$
        and according to Lemma \ref{lm : sum},
        $$H = \frac{(uv)^{|\alpha|+|\beta|}}{(1-u^{|\alpha|}v^{|\alpha|})(1-u^{|\beta|}v^{|\beta|})}+\frac{u^{|\alpha|}v^{|\alpha|}+u^{1+|\alpha|}v^{s_1}}{1-u^{|\alpha|}v^{|\alpha|}}+\frac{u^{|\beta|}v^{|\beta|}+u^{1+|\beta|}v^{s_3}}{1-u^{|\beta|}v^{|\beta|}}$$
        where $S_1=(0,s_1)$ and $S_3=(s_3,0)$ are the first and last vertex of the Newton polygon of $f$. This shows that the pole of $G$ is $-1=\frac{-\alpha\cdot S}{|\alpha|}=\frac{-\beta\cdot S}{|\beta|}$ because $G\in \mathbf{Z}(u,v)\backslash \mathbf{Z}[u,v]$.
    \end{itemize}
    This concludes the proof.
\end{proof}

\bibliographystyle{plain}
\bibliography{bib}

\end{document}